\documentclass[12pt]{amsart}
\usepackage{amsmath}
\usepackage{amsfonts, amssymb, euscript, mathrsfs}
\usepackage{amsthm, upref}
\usepackage{graphicx}
\usepackage[usenames, dvipsnames]{color}

\usepackage{tikz}
\usepackage{enumerate}
\usepackage[margin=0.8in]{geometry}
\usepackage{aliascnt}
\usepackage{hyperref}
\usepackage{verbatim}

\allowdisplaybreaks[4]

\newcommand{\BR}{\mathbb{R}}

\newcommand{\SL}{\sum\limits}

\newcommand{\al}{\alpha}

\newcommand{\de}{\delta}

\newcommand{\MP}{\mathbf P}

\newcommand{\CK}{\mathcal K}

\newcommand{\CP}{\mathcal P}

\newcommand{\oa}{\omega}
\newcommand{\si}{\sigma}

\newcommand{\pa}{\partial}
\renewcommand{\phi}{\varphi}

\newcommand{\eps}{\varepsilon}

\newcommand{\fn}{\mathfrak{n}}
\newcommand{\Ra}{\Rightarrow}
\newcommand{\Lra}{\Leftrightarrow}

\newcommand{\ol}{\overline}

\newcommand{\CM}{\mathcal M}

\newcommand{\norm}[1]{\lVert#1\rVert}
\renewcommand{\comment}[1]{}

\newcommand{\md}{\mathrm{d}}
\newcommand{\CV}{\mathcal V}

\DeclareMathOperator{\SRBM}{SRBM}

\DeclareMathOperator{\mes}{mes}
\DeclareMathOperator{\Int}{int}
\DeclareMathOperator{\dist}{dist}

\begin{document}

\theoremstyle{plain}
\newtheorem{thm}{Theorem}[section]
\newtheorem*{thmnonumber}{Theorem}
\newtheorem{lemma}[thm]{Lemma}
\newtheorem{prop}[thm]{Proposition}
\newtheorem{cor}[thm]{Corollary}
\newtheorem{open}[thm]{Open Problem}

\theoremstyle{definition}
\newtheorem{defn}{Definition}
\newtheorem{asmp}{Assumption}
\newtheorem{notn}{Notation}
\newtheorem{prb}{Problem}

\theoremstyle{remark}
\newtheorem{rmk}{Remark}
\newtheorem{exm}{Example}
\newtheorem{clm}{Claim}

\author{Andrey Sarantsev}

\title[Weak Convergence of Obliquely Reflected Diffusions]{Weak Convergence of Obliquely Reflected Diffusions} 

\address{Department of Statistics and Applied Probability, University of California, Santa Barbara}

\email{sarantsev@pstat.ucsb.edu}

\date{May 2, 2017. Version 27}



\begin{abstract}
Burdzy and Chen (1998) proved results on weak convergence of multidimensional normally reflected Brownian motions. We generalize their work by considering obliquely reflected diffusion processes. We require weak convergence of domains, which is  stronger than convergence in Wijsman topology, but weaker than convergence in Hausdorff topology. 
\end{abstract}

\maketitle

\section{Introduction} Consider a sequence of reflected diffusions $(Z_n)_{n \ge 0}$: for every $n = 0, 1, 2, \ldots$ let $Z_n = (Z_n(t), t \ge 0)$ be a reflected diffusion in $\ol{D}_n$, where $D_n \subseteq \BR^d$ is an open connected subset (bounded or unbounded). When $Z_n(t) \in D_n$, this process is in the interior of its state space, it moves as a diffusion with drift vector field $g_n(\cdot)$ and covariance matrix field $A_n(\cdot)$. When $Z_n$ hits the boundary $\pa D_n$ at a point $z \in \pa D_n$, it is instantaneously reflected inside $D_n$, according to the direction $r_n(z)$. Here, $r_n : \pa D_n \to \BR^d$ is a continuous vector field, defined on the boundary $\pa D_n$. In a more general setting, this boundary can have non-smooth parts (say, the origin for $D_n = (0, \infty)^2$); then the reflection field $r_n$ is defined everywhere on the boundary, except these non-smooth parts. If $r_n(z)$ is the inward unit normal vector to $\pa D_n$ at a point $z \in \pa D_n$, then this reflection is {\it normal} at this point $z$. Otherwise, it is {\it oblique}. We assume that the initial condition is $Z_n(0) = z_n$. 

The main topic of this paper is: When do $Z_n$ weakly converge to $Z_0$ as elements of $C([0, T], \BR^d)$ of continuous functions $[0, T] \to \BR^d$? To establish this convergence, we need convergence of domains, drift vector fields and covariance matrix fields, reflection fields, and initial conditions:
$$
D_n \to D_0,\ g_n \to g_0,\ A_n \to A_0,\ r_n \to r_0,\ z_n \to z_0.
$$
But in which sense do we need to require this convergence? This article provides an answer to this question. The convergence of domains should be in what we call the {\it weak sense}, which is slightly stronger than in the {\it Wijsman topology}, see \cite{Wijsman1966, Wijsman1994}. The convergence of functions poses certain problems, since they are defined on different domains. However, we find a way around this; we define what turns out to be a generalization of locally uniform convergence (and which, in fact, is locally uniform convergence if these functions, say $g_n$, are defined on the same domain). 

Convergence of reflected Brownian motions has been studied in \cite{BurdzyChen1998}
for the case of normal reflection and increasing sequence of domains $D_n \uparrow D_0$. In this article, we study this question in a more general setting: the reflection can be oblique, the concept of convergence $D_n \to D_0$ is more general than $D_n \uparrow D_0$, and we have general diffusion processes (with general drift and covariance fields instead of constant ones) instead of a Brownian motion. 

However, in some sense our conditions are more restrictive: we require the boundary $\pa D_n$ to be smooth, except only a ``small'' subset; in the paper \cite{BurdzyChen1998}, it is only assumed that the boundary is continuous and the domain is bounded. In addition, we assume that the reflection fields $r_n \to r_0$ in a certain sense. In the paper \cite{BurdzyChen1998}, there is no additional assumption that reflection (in their case, normal) fields converge. Last but not least, in our paper the limiting process $Z_0$ should not hit non-smooth parts of the boundary. There are sufficient conditions for this to be true when the domain $D_0$ is a convex polyhedron, see for example \cite{MyOwn3, Wil1987}; see also a related paper \cite{Marshall}. An example of a reflected Brownian motion hitting or not hitting non-smooth parts of the boundary can be found in  Proposition~\ref{prop:non-smooth-result}.  

A related question is an invariance principle for a reflected Brownian motion in a convex polyhedron or, more generally, piecewise smooth domains. This has been studied in \cite{Wil1998b, KangWilliams2007}. See also a recent paper \cite{KR2014} which uses similar techniques to prove well-posedness of a corresponding submartingale problem. We use similar techniques to our paper \cite{MyOwn8}, which deals with penalty method for obliquely reflected diffusions. The difference is that the paper \cite{MyOwn8} approximated an obliquely reflected diffusion by a solution of an SDE without reflection, but with an appropriately chosen drift vector field. The current paper approximates on obliquely reflected diffusion by another obliquely reflected diffusion. 

\subsection{Organization of the paper} Section 2 contains definitions and the main result (Theorem~\ref{thm:main}). In Section 3, we apply these results to reflected Brownian motion in the orthant and in other convex polyhedral domains. Section 4 is devoted to the proof of Theorem~\ref{thm:main}. Section 5 contains results for the case when $D_n \to \BR^d$, that is, the limiting process $Z_0$ is actually a {\it non-reflected} diffusion. 
The Appendix contains some technical lemmata.

\subsection{Notation} For a vector or a matrix $a$, the symbol $a'$ denotes the transpose of $a$. Denote the weak convergence by $\Ra$. Let $C([0, T], \BR^d)$ be the space of all continuous functions $[0, T] \to \BR^d$, with the max-norm. 
 For $d = 1$, we simply write $C[0, T]$. For two vectors $a = (a_1, \ldots, a_d)'$ and $b = (b_1, \ldots, b_d)'$ in $\BR^d$, we denote their dot product by
$a\cdot b = a_1b_1 + \ldots + a_db_d$. The Euclidean norm of $a$ is given by
$\norm{a} = \left[a_1^2 + \ldots + a_d^2\right]^{1/2}$. For $x = (x_1, \ldots, x_d)' \in \BR^d$ and $y = (y_1, \ldots, y_d)' \in \BR^d$, we write $x \ge y$ if $x_i \ge y_i$ for $i = 1, \ldots, d$ and $x > y$ if $x_i > y_i$ for $i = 1, \ldots, d$; similarly for $x \le y$ and $x < y$. For $x \in \BR^d$, $\eps > 0$, let $U(x, \eps) := \{y \in \BR^d\mid \norm{x-y} < \eps\}$ be the $\eps$-neighborhood of $x$. For a point $x \in \BR^d$ and a set $E \subseteq \BR^d$, denote the distance from $x$ to $E$ by $\dist(x, E)$. For a set $E \subseteq \BR^d$ and $r > 0$, denote $U_r(E) = \{x \in \BR^d\mid \dist(x, E) < r\}$. For two sets $E, F \subseteq \BR^d$, denote the distance from $E$ to $F$ by $\dist(E, F)$. For a subset $E \subseteq \BR^d$, we denote the set of its interior points by $\Int E$, and the complement $\BR^d\setminus E$ by $E^c$. We denote its closure by $\ol{E}$. We write $f \in C^r$ for $r$ times continuously differentiable function $f$, defined on some subset of $\BR^d$. We also say that a subset $E$ of $\BR^d$ is $C^r$ if $E$ is an $r$ times continuously differentiable hypersurface in $\BR^d$. The symbol $\mes(E)$ denotes the Lebesgue measure of a set $E$ in $\BR$ or $\BR^d$, depending on the context. The set of all $d\times d$ positive definite symmetric matrices is denoted by $P_d$. Define the {\it modulus of continuity} for a function $f : \BR_+ \to \BR^d$: for $T> 0$ and $\de > 0$, 
$$
\oa(f, [0, T], \de) := \sup\limits_{\substack{t, s \in [0, T]\\ |t - s| \le \de}}\norm{f(t) - f(s)}.
$$

\section{Definitions and the Main Result}

\subsection{Definition of a reflected diffusion} Fix $d \ge 1$, the dimension. Consider a domain (open connected subset) $D \subseteq \BR^d$. Take a function $g : \ol{D} \to \BR^d$ and a matrix-valued function $A : \ol{D} \to P_d$. Let $\si(x) := A^{1/2}(x)$ be the positive definite matrix square root of $A(x)$. Assume that the boundary $\pa D$ is $C^2$ everywhere, except a closed subset $\CV \subseteq \pa D$; that is, $\pa D\setminus\CV$ is $C^2$. The set $\CV$ is called an {\it exceptional set}, or {\it non-smooth parts of the boundary} $\pa D$. For example, if $D = \Int S$, where $S := \BR_+^d$ is an interior of the positive $d$-dimensional orthant, then the boundary $\pa D = \pa S$ consists of $d$ {\it faces}: $S_i := \{x \in S\mid x_i = 0\}$, and $\CV := \cup_{1 \le i < j \le d}\left(D_i\cap D_j\right)$. If $D$ is smooth, or, more precisely, the whole boundary $\pa D$ is $C^2$, then we let $\CV := \varnothing$. 

Denote for $x \in \pa D\setminus\CV$ the {\it inward unit normal vector} by $\fn(x)$. Take a vector field $r : \pa D\setminus\CV \to \BR^d$ such that $r(x)\cdot\fn(x) > 0$. (Without loss of generality, we can assume $r(x)\cdot \fn(x) = 1$; a very short proof of this fact is given in \cite{MyOwn8}.) The function $r$ is called a {\it reflection field}. We note that the set $\CV$ includes, but is not limited to, the parts of the boundary $\pa D$ where it is not $C^2$. It also might include points of the boundary where $\pa D$ is smooth, but the reflection field $r$ is undefined. Slightly abusing the notation, we call the collection of all these points {\it non-smooth parts of the boundary}. 

We would like to define a {\it reflected diffusion} $Z = (Z(t), t \ge 0)$ in $\ol{D}$ with {\it drift coefficient} $g$, {\it covariance matrix} $A$, and {\it reflection field} $r$. 
This is a process that:

(i) behaves as a solution of an SDE with drift coefficient $g$ and covariance matrix $A$, so long as it stays inside $D$;

(ii) when it hits the boundary $\pa D$ at a point $x \in \pa D\setminus\CV$, it reflects according to the reflection vector $r(x)$; if $r(x) = \fn(x)$, this reflection is called {\it normal}, and otherwise it is called {\it oblique}. 

\begin{defn} Take $d$ i.i.d. standard Brownian motions $W_1, \ldots, W_d$ and let $W = (W_1,\ldots, W_d)'$. A continuous adapted process $Z = (Z(t), t \ge 0)$ with values in 
$\ol{D}$ is called a {\it reflected diffusion} in $\ol{D}$, {\it stopped after hitting} $\CV$ with {\it drift vector field} $g$, {\it covariance matrix field} $A$, and {\it reflection field} $r$, {\it starting from} $Z(0) = z_0$, if there exists a real-valued continuous adapted nondecreasing process $l = (l(t), t \ge 0)$ with $l(0) = 0$, such that $l$ can increase only when $Z \in \pa D$, and
\begin{equation}
\label{eq:def-main}
Z(t) = z_0 + \int_0^{t\wedge\tau_{\CV}}g(Z(s))\md s + \int_0^{t\wedge\tau_{\CV}}\si(Z(s))\md W(s) + \int_0^{t\wedge\tau_{\CV}}r(Z(s))\md l(s),\ \ t \ge 0,
\end{equation}
where $\tau_{\CV} := \min\{t \ge 0\mid Z(t) \in \CV\}$, and $\si(x) := A^{1/2}(x)$ is the positive definite symmetric square root of the matrix $A(x)$, for every $x \in \ol{D}$. The process $L(t) := \int_0^{t\wedge\tau_{\CV}}r(Z(s))\md l(s)$, $t \ge 0$, is called the {\it reflection term}. We say this reflected diffusion {\it avoids non-smooth parts of the boundary} if $\tau_{\CV} = \infty$ a.s. 
\label{defn:main}
\end{defn}

We can write~\eqref{eq:def-main} in the differential form:
$$
\md Z(t) = g(Z(t))\md t + \si(Z(t))\md W(t) + r(Z(t))\md l(t),\ \ t < \tau_{\CV}. 
$$
The property that $l$ can increase only when $Z \in \pa D$ can be written formally as
$$
\int_0^{\infty}1(Z(t) \in D)\md l(t) = 0.
$$
There are several conditions for weak or strong existence and uniqueness of this diffusion, discussed in the articles mentioned in the Introduction. In this article, we simply assume that it exists in the weak sense, is unique in law, and does not hit non-smooth parts of the boundary. More precisely, let us state the following assumptions. 

\begin{asmp}
The exceptional set $\CV$ is ``small enough''; namely, for every $x \in \BR^d$ we have: $$
\dist(x, \pa D) = \dist(x, \pa D\setminus\CV).
$$
\label{asmp:boundary}
\end{asmp}

For example, this is true for an orthant $D = (0, \infty)^d$, or a convex polyhedron $D$ (see Section 3). 

\begin{asmp} The reflection field $r : \pa D \setminus\CV \to \BR^d$ is continuous on $\pa D\setminus\CV$. Moreover, as mentioned above, $r(z)\cdot \fn(z) = 1$ for $z \in \pa D\setminus\CV$. 
\end{asmp}

\begin{asmp} The reflected diffusion from Definition~\ref{defn:main} with parameters $g, A, r$, starting from $z_0$, exists and is unique in the weak sense.  
\end{asmp}

A particular case of a reflected diffusion is {\it reflected Brownian motion}, when the drift coefficient $g(x)$ and the covariance matrix $A(x)$ do not depend on $x$: $g(x) \equiv g$, and $A(x) \equiv A$. An example of a reflected Brownian motion hitting or not hitting non-smooth parts of the boundary is given in Section 3, Proposition~\ref{prop:non-smooth-result}.

\subsection{Weak convergence of domains} For each $n = 0, 1, 2, \ldots$, define the function $\phi_n : \BR^d \to \BR$ to be the {\it signed distance} to $\pa D_n$:
$$
\phi_n(x) := 
\begin{cases}
\dist(x, \pa D_n),\ \ x \in D_n;\\
0,\ \ x \in \pa D_n;\\
-\dist(x, \pa D_n),\ \ x \in \BR^d\setminus\ol{D_n}.
\end{cases}
$$

\begin{defn} We say that the sequence of domains $(D_n)_{n \ge 1}$ {\it converges weakly} to the domain $D_0$ in $\BR^d$, and write $D_n \Ra D_0$, if $\phi_n(x) \to \phi_0(x)$ for every $x \in \BR^d$. 
\label{def:conv-domain}
\end{defn}

There are other well-known concepts of set convergence in $\BR^d$. 

\begin{defn} Take subsets $E_n \subseteq \BR^d,\ n = 0, 1, 2, \ldots$ We say that $E_n \to E_0$ in {\it Wijsman topology} if $\dist(x, E_n) \to \dist(x, E_0)$ for all $x \in \BR^d$. If this convergence is uniform for $x \in \BR^d$, then $E_n \to E_0$ in {\it Hausdorff topology}. An equivalent definition of Hausdorff convergence is through {\it Hausdorff distance}, which is defined for $A, B \subseteq \BR^d$ as follows:
$$
d_H(A, B) = \inf\{\eps > 0\mid A \subseteq U_{\eps}(B)\ \ \mbox{and}\ \ B \subseteq U_{\eps}(A)\}.
$$
\end{defn}

For Wijsman convergence, we can substitute $E_n$ by their closures, because 
$$
\dist(x, E_n) \equiv \dist(x, \ol{E}_n).
$$

There are equivalent definitions of Hausdorff convergence, distance and topology. We refer the reader to the book \cite{MunkresBook}. For Wijsman convergence, see the articles \cite{Wijsman1966, Wijsman1994}. In a sense, both Wijsman convergence and weak convergence are ``local'' analogues of Hausdorff convergence, just as locally uniform convergence of functions with respect to uniform convergence. Let us state a few elementary properties of Wijsman and weak convergence, with the proofs postponed until Appendix.

\begin{lemma}
\label{lemma:Wijsman}
Suppose $E_n \to E_0$ in Wijsman topology. Then:

(i) $\dist(x, E_n) \to \dist(x, E_0)$ uniformly on every compact subset $\CK \subseteq \BR^d$;

(ii) if $x_n \in E_n$ and $x_n \to x_0$, then $x_0 \in \ol{E}_0$.
\end{lemma}

\begin{lemma}
\label{lemma:equiv}
The following statements for domains $D_n,\ n = 0, 1, 2, \ldots$ are equivalent:

(i) $D_n \Ra D_0$;

(ii) $D_n \to D_0$ and $D_n^c \to D_0^c$ in Wijsman topology;

(iii) $\phi_n(x) \to \phi_0(x)$ uniformly on every compact subset $\CK \subseteq \BR^d$;

(iv) for every compact subset $\CK \subseteq \BR^d$ and a sequence $(\eps_n)_{n \ge 1}$ with $\eps_n \to 0$, we have: 
$$
\max\limits_{\substack{x_n, x_0 \in \CK\\ \norm{x_n - x_0} \le \eps_n}} \left|\phi_n(x_n) - \phi_0(x_0)\right| \to 0\ \ \mbox{as}\ \ n \to \infty;
$$

(v) for every $T > 0$ and every sequence $(f_n)_{n \ge 1}$ of functions $f_n : [0, T] \to \BR^d$ which converges uniformly on $[0, T]$ to a continuous function $f_0 : [0, T] \to \BR^d$, we have: $\phi_n(f_n(\cdot)) \to \phi_0(f_0(\cdot))$ uniformly on $[0, T]$;

(vi) $\pa D_n \to \pa D_0$ in Wijsman topology, and, in addition, 
\begin{equation}
\label{eq:inclusion}
D_0 \subseteq \varliminf\limits_{n \to \infty}D_n\ \  \mbox{and}\ \  \varlimsup\limits_{n \to \infty}\ol{D}_n \subseteq \ol{D}_0;
\end{equation}

(vii) if $x_{n_k} \in \pa D_{n_k}$ and $x_{n_k} \to x_0$ for some subsequence $(n_k)_{k \ge 1}$, then $x_0 \in \pa D_0$, and~\eqref{eq:inclusion} holds. 
\end{lemma}


\begin{cor} Assume $D_n \Ra D_0$. 

(i) Take a sequence $(x_k)_{k \ge 1}$. If $x_k \in \ol{D_{n_k}}$ for some subsequence $(n_k)_{k \ge 1}$, and $x_k \to x_0$, then $x_0 \in \ol{D_0}$. If $x_k \in D^c_{n_k}$, and $x_k \to x_0$, then $x_0 \in D_0^c$. If $x_k \to x_0$, and $x_k \in \pa D_{n_k}$, then $x_0 \in \pa D_0$. 

(ii) For every compact subset $\CK \subseteq D_0$, there exists $n_0$ such that for $n > n_0$, we have: $\CK \subseteq D_n$. For every compact subset $\CK \subseteq \ol{D}_0^c$, there exists $n_0$ such that for $n > n_0$, we have: $\CK \subseteq \ol{D}_n^c$. 
\label{cor:simple}
\end{cor} 

When $(D_n)_{n \ge 1}$ is a monotone sequence, this concept of convergence can be simplified. 

\begin{lemma} If $D_n \uparrow D_0$ or $\ol{D}_n \downarrow \ol{D}_0$, then $D_n \Ra D_0$. 
\label{lemma:monotone} 
\end{lemma}

The following lemma provides comparison of convergence modes. 

\begin{lemma} (i) Weak convergence $D_n \Ra D_0$ is stronger than Wijsman convergence.

(ii) Weak convergence $D_n \Ra D_0$ is weaker than Hausdorff convergence. 

(iii) $D_n \to D_0$ in Hausdorff topology if and only if $\phi_n(x) \to \phi_0(x)$ uniformly on the whole $\BR^d$. 
\label{lemma:comparison-conv}
\end{lemma}

\begin{exm} Fix $d \ge 2$, the dimension. Let $e_1 = (1, 0, \ldots, 0)' \in \BR^d$. Consider a sequence $D_n := U(ne_1, n)$ of open balls of radius $n$ centered at $ne_1$. This is an increasing sequence: $D_n \subseteq D_{n+1}$. It is easy to see that $D_n \uparrow D_0 = \{x \in \BR^d\mid x_1 > 0\}$. By Lemma~\ref{lemma:monotone}, $D_n \Ra D_0$.
\end{exm}

\begin{exm} Take a sequence $D_n = U(x_n, a_n)$ of open discs in $\BR^d$. Then $D_n \Ra D_0$ if and only if $x_n \to x_0$ and $a_n \to a_0$. Indeed, 
$\phi_n(x) \equiv a_n - \norm{x - x_n}$, so the ``if'' part is obvious. Let us show the ``only if'' part. Assume $D_n \Ra D_0$. Take an arbitrarily small $\eps > 0$. Then by Corollary~\ref{cor:simple}, for $\CK = \ol{U(x_0, a_0 - \eps)} \subseteq D_0$, there exists $n_0$ such that for $n > n_0$ we have: $\CK \subseteq D_n = U(x_n, a_n)$. But if $\ol{U(y_1, a_1)} \subseteq U(y_2, a_2)$, then $a_1 < a_2$ and $\norm{y_1 - y_2} \le a_2 - a_1$. Therefore, 
\begin{equation}
\label{eq:discintodisc}
a_0 - \eps < a_n\ \ \mbox{and}\ \ \norm{x_n - x_0} \le a_n - a_0 + \eps\ \ \mbox{for}\ \ n > n_0.
\end{equation}
We can take arbitrarily small $\eps > 0$. From the first comparison in~\eqref{eq:discintodisc}, 
\begin{equation}
\label{eq:limsup}
\varliminf\limits_{n \to \infty}a_n \ge a_0.
\end{equation}
Similarly, taking $\CK = U(0, N)\setminus U(x_0, a_0 + \eps)$ for large $N$ and small $\eps > 0$, we conclude: $\CK \subseteq \ol{D}_0^c$, and so $\CK \subseteq \ol{D}_n^c$ for large enough $n$. Therefore, $a_n \le a_0 + \eps$. This leads to the conclusion that 
\begin{equation}
\label{eq:liminf}
\varlimsup\limits_{n \to \infty}a_n \le a_0.
\end{equation}
Combining~\eqref{eq:limsup} and~\eqref{eq:liminf}, we get: $a_n \to a_0$. Now, from the second comparison in~\eqref{eq:discintodisc} we have: because $a_n \to a_0$ and $\eps > 0$ is arbitrarily small, $x_n \to x_0$. 
\end{exm}

\begin{exm} Take a sequence $(f_n)_{n \ge 1}$ of smooth functions $\BR_+^{d-1} \to \BR$ such that $f_n \to 0$ locally uniformly and $f_n(0) = 0$. For $i = 1, \ldots, d$ and $x = (x_1, \ldots, x_d)'$, we let
$$
\hat{x}_i = (x_1, \ldots, x_{i-1}, x_{i+1}, \ldots, x_d)' \in \BR^{d-1}.
$$
Now, define the following sequence of domains:
$$
D_n = \{x \in \BR^d\mid x_i > f_n(\hat{x}_i),\ i = 1, \ldots, d\}.
$$
Then $D_n \Ra D_0 = (0, \infty)^d$. The proof is similar to that of Theorem~\ref{thm:SRBM} (i), (ii) below.
\end{exm}

\subsection{Main result} Consider a sequence $(D_n)_{n \ge 1}$ of domains in $\BR^d$. Let $\CV_n$ be non-smooth parts of the boundary for $D_n$. For each $n = 0, 1, 2, \ldots$ take a reflected diffusion $Z_n$ in $\ol{D}_n$ with drift vector $g_n$, covariance matrix $A_n$, and reflection field $r_n$, starting from $z_n = Z_n(0)$. Suppose that for every $n = 0, 1, 2, \ldots$, this reflected diffusion $Z_n$ satisfies Assumptions 1-3.

The main question of this paper is:

\begin{center}
Under what assumptions on $g_n, A_n, r_n, z_n, D_n$, do we have: 
$$
Z_n \Ra Z_0\ \ \mbox{weakly in}\ \ C([0, T], \BR^d)\ \ \mbox{for every}\ \ T > 0\, ?
$$
\end{center}

First, we need the domains $D_n$ to converge to $D_0$ in some sense. We already defined an appropriate concept of weak convergence earlier. We also need to have
$$
g_n \to g_0,\ \ A_n \to A_0,\ \ r_n \to r_0
$$
uniformly in some sense. But these functions are defined on different subsets of $\BR^d$. A natural way to define convergence is as follows.

\begin{defn} Take functions $f_n : E_n \to \BR^p$, $n = 0, 1, 2, \ldots$ where $E_n \subseteq \BR^d$, and $p \ge 1$ is some dimension. We say that $f_n \to f_0$ {\it locally uniformly}, and write $f_n \Ra f_0$, if one of these two equivalent statements is true:

(i) for every subsequence $(n_k)_{k \ge 1}$ and any sequence $(z_{n_k})_{k \ge 1}$ such that $z_{n_k} \in E_{n_k}$ and $z_{n_k} \to z_0 \in E_0$ we have: $f_{n_k}(z_{n_k}) \to f_0(z_0)$;

(ii) for every $T > 0$, and for every subsequence $(n_k)_{k \ge 1}$ and any sequence $(x_{n_k})_{k \ge 1}$ of continuous functions $[0, T] \to \BR^d$ such that $x_n(t) \in E_n$ for all $n = 0, 1, \ldots$ we have: 
$$
\mbox{if}\ \ x_{n_k}(t) \to x_0(t)\ \ \mbox{uniformly on}\ \ [0, T],
$$
$$
\mbox{then}\ \ f_{n_k}(x_{n_k}(t)) \to f_0(x_0(t))\ \ \mbox{uniformly on}\ \ [0, T].
$$
\label{defn:uniform-conv}
\end{defn}

\begin{lemma} These two definitions (i) and (ii) of locally uniform convergence are indeed equivalent, if $f_0$ is continuous on $E_0$. 
\label{lemma:loc-unif-conv}
\end{lemma}

The proof of Lemma~\ref{lemma:loc-unif-conv} is postponed until Appendix. 

\begin{rmk} In the case $E_n = E_0$, if the function $f_0$ is continuous, then $f_n \Ra f_0$ is equivalent to the locally uniform convergence on $E_0$ in the usual sense (that is, uniform convergence on $E_0\cap\CK$ for every compact set $\CK \subseteq \BR^d$). 
\end{rmk}

\begin{rmk} Note that $A_n \Ra A_0$ if and only if $\si_n(x) := A_n^{1/2}(x) \Ra \si_0(x) := A_0^{1/2}(x)$. The ``if'' part follows from the obvious fact that the operation of taking the square of a matrix is continuous. The ``only if'' part follows from the fact that the operation of taking a symmetric positive definite square root of a symmetric positive definite matrix is also continuous, see for example \cite{HornBook}. 
\label{rmk:sqrt}
\end{rmk}

Now comes the main result of this paper. 

\begin{thm}
\label{thm:main}
Take $Z_n$ for $n = 0, 1, 2, \ldots$ as described above. Assume each $Z_n,\ n = 0, 1, 2, \ldots$ satisfies Assumptions 1-3. Suppose that $g_0, A_0, r_0$ are locally bounded, and $Z_0$ does not hit non-smooth parts of the boundary. Assume that
$$
D_n \Ra D_0,\ \ g_n \Ra g_0,\ \ A_n \Ra A_0,\ \ r_n \Ra r_0,\ \ z_n \to z_0.
$$
Also, assume that at least one of the following conditions (a) or (b) holds true: 

(a) for all $n \ge n_0$, the process $Z_n$ does not hit non-smooth parts $\CV_n$ of the boundary $\pa D_n$;

(b) for every compact set $\CK \subseteq \BR^d$, we have:
\begin{equation}
\label{eq:condition-b}
\lim\limits_{n \to \infty}\max\limits_{x \in \CV_n\cap\CK}\dist(x, \CV_0) = 0. 
\end{equation}
Then $Z_n \Ra Z_0$ weakly in $C([0, T], \BR^d)$ for every $T > 0$.
\end{thm}

The following is a necessary and sufficient condition for~\eqref{eq:condition-b}. 

\begin{lemma} Condition (b) from Theorem~\ref{thm:main} holds if and only if for every sequence $(x_{n_k})_{k \ge 1}$ with $x_{n_k} \in \CV_{n_k}$ and $x_{n_k} \to x_0$ we have: $x_0 \in \CV_0$. In particular, we can apply Lemma~\ref{lemma:Wijsman} (i) and conclude: condition (b) holds if $\CV_n \to \CV_0$ in Wijsman topology. 
\label{lemma:conv-of-non-smooth}
\end{lemma}

If all domains $D_0, D_1, D_2, \ldots$ are the same, then we can restate this main result as follows. 

\begin{cor}
\label{cor:same-domain}
Assume $D_n = D$ for $n = 0, 1, 2, \ldots$, where $D$ has non-smooth parts of the boundary $\CV$. Suppose $r_n \to r_0$ locally uniformly on $\pa D\setminus\CV$, 
and $g_n \to g_0$, $\si_n \to \si_0$ locally uniformly on $\ol{D}\setminus\CV$. Assume $z_n \to z_0$. Finally, assume $Z_0$ does not hit $\CV$. Then 
$$
Z_n \Ra Z_0\ \ \mbox{weakly in}\ \ C([0, T], \BR^d)\ \ \mbox{for every}\ \ T > 0\,.
$$
\end{cor}

\section{Semimartingale Reflected Brownian Motion in a Convex Polyhedron}

\subsection{Definitions} An open convex polyhedron $D$ is defined as follows. Fix $m \ge 1$, the number of edges. Let $\fn_1, \ldots, \fn_m \in \BR^d$ be unit vectors, and let $b_1, \ldots, b_m \in \BR$ be real numbers. The domain $D$ is defined as
\begin{equation}
\label{eq:def-D}
D = \{x \in \BR^d\mid x\cdot\fn_i > b_i,\ \ i = 1, \ldots, m\}
\end{equation}
We assume that $D \ne \varnothing$, and for each $j = 1, \ldots, m$, we have:
$$
\{x \in \BR^d\mid x\cdot \fn_i > b_i,\ \ i = 1, \ldots, m,\ \ i \ne j\} \ne D.
$$
In this case, the {\it edges of $D$}: $D_i = \{x \in \ol{D}\mid \fn_i\cdot x = b_i\}$, $i = 1, \ldots, m$, are $(d-1)$-dimensional. The vector $\fn_i$ is the inward unit normal vector to the face $D_i$, for each $i = 1, \ldots, m$. The following subset of the boundary is called {\it non-smooth parts of the boundary}, and in our notation, it plays the role of the exceptional set $\CV$:
$$
\CV = \bigcup\limits_{1 \le i < j \le m}\left(D_i\cap D_j\right).
$$
We should note that $\CV$ satisfies Assumption~\ref{asmp:boundary}. 
The closure $\ol{D}$ of $D$ is called a {\it closed convex polyhedron}. 
In the sequel, we sometimes simply refer to $D$ or $\ol{D}$ as a {\it convex polyhedron}, if it is obvious from the context which one we are referring to. 

Now, let us define an {\it SRBM in the polyhedron} $\ol{D}$, with {\it drift vector} $\mu \in \BR^d$, {\it covariance matrix} $A$, and a $d\times m$-reflection matrix $R$. This is a continuous adapted process $Z = (Z(t), t \ge 0)$, which can be represented as 
$$
Z(t) = W(t) + RL(t),\ \ t \ge 0.
$$
Here, $W = (W(t), t \ge 0)$ is a $d$-dimensional Brownian motion with drift vector $\mu$ and covariance matrix $A$, and $L = (L_1, \ldots, L_m)'$, where for each $i = 1, \ldots, m$, $L_i = (L_i(t), t \ge 0)$ is a real-valued continuous nondecreasing adapted process with $L_i(0) = 0$, which can increase only when $Z \in D_i$. This is denoted by $Z = \SRBM^d(\ol{D}, R, \mu, A)$. This is a process which reflects on each face $D_i$, $i = 1, \ldots, m$, according to the vector $r_i$, the $i$th column of the reflection matrix $R$. A particular case is an {\it SRBM in the orthant} $S = \BR^d_+$, when $m = d$, $\fn_i = e_i$ is the standard $i$th unit vector in $\BR^d$, and $b_i = 0$. Then $R$ is a $d\times d$-matrix, and the process $Z$ is denoted by $\SRBM^d(R, \mu, A)$. 

An SRBM in a convex polyhedron, and, in particular, in the orthant, was a subject of extensive study over the past few decades. Existence and uniqueness results (weak and strong) are proved in \cite{DW1995, HR1981a, RW1988, TW1993}. For an SRBM in the orthant, see the survey \cite{Wil1995}. 

An SRBM in a convex polyhedron fits into our general framework as follows: 
define the reflection field $r : \pa D \setminus\CV \to \BR^d$ to be
$r(x) = r_i$ for $x \in D_i\setminus\CV$, $i = 1, \ldots, m$. This function is continuous on $\pa D\setminus\CV$. Sufficient conditions when an $\SRBM^d(\ol{D}, R, \mu, A)$ does not hit non-smooth parts of the boundary $\CV$ are known: see \cite{MyOwn3, Wil1987}. Let us give an example.

\begin{prop}
\label{prop:non-smooth-result}
Consider a reflected Brownian motion $\SRBM^d(R, \mu, A)$ with $A = (a_{ij})_{i, j = 1,\ldots, d}$, and with $R = (r_{ij})_{i, j = 1, \ldots, d}$ having $r_{ii} = 1$, $i = 1, \ldots, d$; $r_{ij} \le 0$, $i \ne j$; and the spectral radius of $I_d - R$ is strictly less than $1$. Then this SRBM a.s. does not hit non-smooth parts of the boundary if and only if
$$
r_{ij}a_{jj} + r_{ji}a_{ii} \ge 2a_{ij},\ \ i, j = 1, \ldots, d.
$$
\end{prop}

\subsection{Main Result} The following result is a corollary of Theorem~\ref{thm:main}. 

\begin{thm}
Take $m$ sequences of real numbers $(b_{i, n})_{n \ge 0}$, $i = 1, \ldots, m$. Take $m$ sequences of unit vectors in $\BR^d$: $(\fn_{i, n})_{n \ge 0}$, $i = 1, \ldots, m$.
Assume that 
\begin{equation}
\label{eq:limit-parameters-SRBM}
\fn_{i, n} \to \fn_{i, 0},\ \ b_{i, n} \to b_{i, 0},\ \ n \to \infty,\ \ \mbox{for each}\ \ i = 1, \ldots, m.
\end{equation}
Consider a sequence $(D_n)_{n \ge 0}$ of convex polyhedra given by
$$
D_n = \{x \in \BR^d\mid x\cdot \fn_{i, n} > b_{i, n},\ \ i = 1, \ldots, m\}.
$$
Take a sequence of positive definite symmetric $d\times d$ matrices $(A_n)_{n \ge 0}$ such that $A_n \to A_0$ as $n \to \infty$. Take a sequence $(g_n)_{n \ge 0}$ in $\BR^d$ such that $g_n \to g_0$ as $n \to \infty$. Take a sequence of reflection matrices $(R_n)_{n \ge 0}$ such that $R_n \to R_0$. Assume that for every $n \ge 0$, the process
$Z_n := \SRBM^d(\ol{D}_n, R_n, g_n, A_n)$, starting from $Z_n(0) = z_n \in \ol{D}_n$, exists in the weak sense and is unique in law, and $z_n \to z_0$. Assume also that the process $Z_0$ does not hit non-smooth parts of the boundary $\pa D_0$. Then 
$$
Z_n \Ra Z_0\ \ \mbox{weakly in}\ \ C([0, T], \BR^d)\ \ \mbox{for every}\ \ T > 0\,.
$$
\label{thm:SRBM}
\end{thm}

The proof is postponed until the next subsection. Let us give an application. 

\begin{exm} Consider a fixed convex polyhedron $D \subseteq \BR^d$. Let 
$$
\CP := \{(R, A)\mid \SRBM^d(\ol{D}, R, \mu, A)\ \ \mbox{does not hit non-smooth parts of the boundary}\}.
$$
This definition makes sense because of the following fact: The property that an $\SRBM^d(\ol{D}, R, \mu, A)$ does not hit non-smooth parts of the boundary is independent of the starting point $z \in D$ and of the drift vector $\mu$. The proof of this independence statement is similar to that of \cite[Proposition 3.3]{MyOwn3}.
From Theorem~\ref{thm:SRBM}, we can conclude that the process $\SRBM^d(\ol{D}, R, \mu, A)$, starting from $z \in D$, is continuous as an element of $C([0, T], \BR^d)$, for every $T > 0$, on the set 
$$
\{(z, R, \mu, A)\mid z \in D,\ (R, A) \in \CP,\ \mu \in \BR^d\}.
$$
\end{exm} 

\subsection{Proof of Theorem~\ref{thm:SRBM}} We need to show that: 

(i) $D_n \Ra D_0$;

(ii) the condition (b) from Theorem~\ref{thm:main} is satisfied;

(iii) $r_n \Ra r_0$. 

\medskip

{\it Proof of (i).} We use Lemma~\ref{lemma:equiv} (vii). Take a subsequence 
$(n_k)_{k \ge 1}$ and let $x_{n_k} \in \pa D_{n_k}$ be such that $x_{n_k} \to x_0$.
Let us show that $x_0 \in \pa D_0$. The boundary $\pa D_n$ for every $n$ consists of $m$ parts:
$$
\pa D_n = \bigcup\limits_{i=1}^mD_{n, i},\ \ D_{n, i} := \{x \in \BR^d\mid \fn_{i, n}\cdot x = b_{i, n},\ \ \fn_{j, n}\cdot x \ge b_{j, n},\ \ j = 1, \ldots, m,\ \ j \ne i\}.
$$
By the pigeonhole principle, there exists an $i_0 \in \{1, \ldots, m\}$ and a subsequence $(n'_k)_{k \ge 1} \subseteq (n_k)_{k \ge 1}$ such that 
$x_{n'_k} \in D_{n'_k, i_0}$. 
That is,
$$
\fn_{i_0, n'_k}\cdot x_{n'_k} = b_{i_0, n'_k},\ \ \fn_{j, n'_k}\cdot x_{n'_k} \ge b_{j, n'_k},\ \ j = 1, \ldots, m,\ \ j \ne i_0.
$$
Letting $k \to \infty$, we have: $\fn_{i_0, 0}\cdot x_0 = b_{i_0, 0}$, and $\fn_{j, 0}\cdot x_0 \ge b_{j, 0}$, $j = 1, \ldots, m$, $j \ne i_0$. Therefore, $x_0 \in D_{0, i_0} \subseteq \pa D_0$. Now, let us show~\eqref{eq:inclusion}. Take $x_0 \in D_0$. Then $\fn_{i, 0}\cdot x_0 > b_{i, 0}$, for $i = 1, \ldots, m$. 
From~\eqref{eq:limit-parameters-SRBM}, we get: there exists $n_0$ such that for $n > n_0$ we have: $\fn_{i, n}\cdot x_0 > b_{i, n}$, $i = 1, \ldots, m$. So $x_0 \in D_n$ for $n > n_0$; therefore, $x_0 \in \varliminf D_n$. Similarly, if $x_0 \in \ol{D}_0^c$, then there exists a $j \in \{1, \ldots, m\}$ such that $\fn_{j, 0}\cdot x_0 < b_{j, 0}$. From~\eqref{eq:limit-parameters-SRBM}, we get: there exists $n_0$ such that for $n > n_0$ we have: $\fn_{j, n}\cdot x_0 < b_{j, n}$. Therefore, $x_0 \in \ol{D}_n^c$ for $n > n_0$; so $x_0 \in \varliminf \ol{D}_n^c$. This completes the proof of ~\eqref{eq:inclusion}.

\medskip

{\it Proof of (ii).} We use Lemma~\ref{lemma:conv-of-non-smooth}. The domain $D_n$ has non-smooth parts of the boundary
$$
\CV_n := \bigcup\limits_{1 \le i < j \le m}D_{n, i, j},
$$
where we denote
$$
D_{n, i, j} := \left\{x \in \BR^d\mid x\cdot\fn_{i, n} = b_{i, n},\ \ x\cdot\fn_{j, n} = b_{j, n},\ \ x\cdot\fn_{q, n} \ge b_{q, n},\ q \ne i, j\right\}.
$$
Now, take a sequence $(x_{n_k})_{k \ge 1}$ with $x_{n_k} \in \CV_{n_k}$ and show that if $x_{n_k} \to x_0$, then $x_0 \in \CV_0$. By the pigeonhole principle, there exist a subsequence $(n'_k)_{k \ge 1}$ and $1 \le i < j \le m$ such that $x_{n'_k} \in D_{n'_k, i, j}$. Therefore,
$$
x_{n'_k}\cdot\fn_{i, n'_k} = b_{i, n'_k},\ \ x_{n'_k}\cdot\fn_{j, n'_k} = b_{j, n'_k},\ \ x_{n'_k}\cdot\fn_{q, n'_k} \ge b_{q, n'_k},\ q \ne i, j.
$$
Letting $k \to \infty$, we get:
$$
x_0\cdot\fn_{i, 0} = b_{i, 0},\ \ x_0\cdot\fn_{j, 0} = b_{j, n'_k},\ \ x_0\cdot\fn_{q, n'_k} \ge b_{q, n'_k},\ q \ne i, j.
$$
Therefore, $x_0 \in D_{0, i, j} \subseteq \CV_0$. This completes the proof of (ii).

\medskip

{\it Proof of (iii).} Take $x_n \in \pa D_n\setminus\CV_n,\, n = 0, 1, 2, \ldots$ such that $x_n \to x_0$. We need to prove that $r_n(x_n) \to r_0(x_0)$. Let us show that for every subsequence $(n_k)_{k \ge 1}$, there exists a subsequence $(n'_k)_{k \ge 1} \subseteq (n_k)_{k \ge 1}$ such that 
$$
r_{n'_k}\bigl(x_{n'_k}\bigr) \to r_0(x_0).
$$
Indeed, by the pigeonhole principle, there exists a $j \in \{1, \ldots, m\}$ and a subsequence 
$(n'_k)_{k \ge 1}$ such that $x_{n'_k} \in D_{n'_k, j}$. Then, as discussed in the proof of (i) above, $x_0 \in D_{0, j}$. Denote the $j$th column of $R_n$ by $r_{n, j}$. Then 
$r_n(x) \equiv r_{n, j}$ for $x \in D_{n, j}$, by definition of a reflection field for an SRBM in a convex polyhedron. Now, $r_{n'_k}\bigl(x_{n'_k}\bigr) = r_{n'_k, j} \to r_{0, j} = r_0(x_0)$, because $R_n \to R_0$. This completes the proof.

\section{Proof of Theorem~\ref{thm:main}}

\subsection{Outline of the proof} For the rest of this section, fix a time horizon $T > 0$. The first step is {\it localization}. Consider a compact set 
$\CK \subseteq \BR^d\setminus\CV$ such that $z_0 \in \Int\CK$. Let 
$$
\tau_{\CK, n} := \inf\{t \ge 0\mid Z_n(t) \notin \Int\CK\},\ \ n = 0, 1, 2, \ldots
$$
Let $Z^{\CK}_n(t) \equiv Z_n(t\wedge\tau_{\CK, n})$. We say that a continuous adapted process $\zeta = (\zeta(t), t \in [0, T])$ {\it behaves as $\ol{Z}_0$ until it exits $\Int\CK$} if for the stopping time
$$
\ol{\tau}_{\CK, 0} := \inf\{t \ge 0\mid \zeta(t) \notin \CK\},
$$
the process $\zeta\left(\cdot\wedge\ol{\tau}_{\CK, 0}\right)$ has the same law as $Z^{\CK}_0$. The following lemma was, in fact, already proved as Lemma 4.1 in \cite{MyOwn8}. 

\begin{lemma}
\label{lemma:local}
Assume that for every compact subset $\CK$ as above every weak limit point of the sequence $(Z^{\CK}_n)_{n \ge 1}$ in $C([0, T], \BR^d)$ behaves as $\ol{Z}_0$ until it exits $\Int\CK$. Then the conclusion of Theorem~\ref{thm:main} is true.
\end{lemma}

\begin{rmk} If either (a) or (b) holds, then for every compact set $\CK \subseteq \BR^d\setminus\CV_0$ there exists $n_{\CK}$ such that for $n \ge n_{\CK}$, we have: 
$Z_n^{\CK}$ does not hit $\CV_n$. Indeed, if (a) holds true, then there is nothing to prove. If (b) holds true, then $\dist(\CK, \CV_0) := \eps_0 > 0$, and there exists $n_{\CK}$ such that for $n \ge n_{\CK}$, we have: 
$$
\max\limits_{x \in \CV_n\cap\CK}\dist(x, \CV_0) < \eps_0.
$$
In this case, for every $n \ge n_{\CK}$ we have: $\CK\cap\CV_n = \varnothing$. Therefore, $Z^{\CK}_n(t) \notin \CV_n$ for these $n$ and for $t \in [0, T]$. 
\end{rmk}

The rest of the proof of Theorem~\ref{thm:main} tracks the proofs from the paper \cite{MyOwn8}.  

\begin{lemma}
\label{lemma:tight-phi}
The sequence $(\phi_0(Z^{\CK}_n(\cdot)))_{n \ge 1}$ is tight in $C[0, T]$. 
\end{lemma}

Now, we can split $Z^{\CK}_n$ into two components:
\begin{equation}
\label{eq:Z=V+L}
Z^{\CK}_n(t) \equiv Z_n\left(t\wedge\tau_{\CK, n}\right) = V_n(t) + L_n(t),
\end{equation}
where for $n  =1, 2, \ldots$ and $t \in [0, T]$ we define:
$$
W_n^{\CK}(t) = W_n\left(t\wedge\tau_{\CK, n}\right),
$$
$$
V_n(t) := z_n + \int_0^{t\wedge\tau_{\CK, n}}g_n(Z_n(s))\md s + \int_0^{t\wedge\tau_{\CK, n}}\si_n(Z_n(s))\md W_n(s),
$$
$$
L_n(t) := \int_0^{t\wedge\tau_{\CK, n}}r_n(Z_n(s))\md l_n(s),\ \ \mbox{and}\ \ l_n^{\CK}(t) = l_n\left(t\wedge\tau_{\CK, n}\right),
$$
and $l_n$ is the process $l$ from Definition~\ref{defn:main} for the reflected diffusion $Z_n$ in place of $Z$. 

\begin{lemma} 
\label{lemma:tight-V}
The sequence $(V_n)_{n \ge 1}$ is tight in $C[0, T]$. 
\end{lemma}

\begin{lemma}
\label{lemma:tight-l}
The sequence $(l_n)_{n \ge 1}$ is tight in $C[0, T]$. 
\end{lemma}

\begin{lemma}
\label{lemma:tight-L}
The sequence $(L_n)_{n \ge 1}$ is tight in $C[0, T]$. 
\end{lemma}

The sequence $(W_n)_{n \ge 1}$ of Brownian motions is obviously tight in $C([0, T], \BR^d)$ (all Brownian motions $W_n$ have the same distribution). Because each $W_n^{\CK}$ is a Brownian motion $W_n$ stopped when it exits $\CK$, the sequence $(W_n^{\CK})_{n \ge 1}$ is also tight in $C([0, T], \BR^d)$. Using Lemmata~\ref{lemma:tight-V}, ~\ref{lemma:tight-l} and~\ref{lemma:tight-L}, take a weak limit point $\left(\ol{V}, \ol{L}, \ol{l}, \ol{W}\right)'$ of the sequence 
$$
\left(V_n,  l^{\CK}_n, L^{\CK}_n, W^{\CK}_n\right)'.
$$
We have: for some subsequence $(n_k)_{k \ge 1}$, 
\begin{equation}
\label{eq:conv}
\left(V_{n_k}, l^{\CK}_{n_k}, L^{\CK}_{n_k}, W^{\CK}_{n_k}\right) \Ra \left(\ol{V}, \ol{L}, \ol{l}, \ol{W}\right)'.
\end{equation}
By Skorohod representation theorem, see for example \cite[Chapter 1]{IWBook}, we can assume that the convergence is a.s. on a common probability space. From~\eqref{eq:Z=V+L}, we have:
$$
\ol{Z}(t) := \ol{V}(t) + \ol{L}(t) = \lim\limits_{k \to \infty}Z^{\CK}_{n_k}(t),
$$
where the convergence is uniform on $[0, T]$. 

\begin{lemma} The process $\ol{W}$ is a $d$-dimensional Brownian motion (with zero drift vector and identity covariance matrix), at least until the stopping time
$\ol{\tau}_{\CK} := \inf\{t \ge 0\mid \ol{Z}(t) \notin \Int\CK\}$. In addition, $\ol{\tau}_{\CK} \le \varliminf\limits_{k \to \infty}\tau_{\CK, k}$ a.s. 
\label{lemma:rep-W}
\end{lemma}

Lemma~\ref{lemma:rep-W} was proved as Lemma 4.5 in \cite{MyOwn8}. 

\begin{lemma}
\label{lemma:rep-V}
For $t \in [0, \ol{\tau}_{\CK}]$, 
$$
\ol{V}(t) = z_0 + \int_0^tg_0(\ol{Z}(s))\md s + \int_0^t\si_0(\ol{Z}(s))\md \ol{W}(s).
$$
\end{lemma}

Now, let us state two lemmata which deal with the reflection terms. 

\begin{lemma}
\label{lemma:rep-l}
On the interval $[0, \ol{\tau}_{\CK}]$, the process $\ol{l}$ is continuous, nondecreasing, can increase only when $\ol{Z} \in \pa D_0$, and $\ol{l}(0) = 0$. 
\end{lemma}

\begin{lemma}
\label{lemma:rep-L}
For $t \in [0, T]$, 
$$
\ol{L}(t) = \int_0^tr_0(\ol{Z}(s))\md \ol{l}(s).
$$
\end{lemma}

Now, let us complete the proof of Theorem~\ref{thm:main}. Take a sequence $(m_k)_{k \ge 1}$. As in~\eqref{eq:conv}, there exists a subsequence $(n_k)_{k \ge 1}$ such that~\eqref{eq:conv} holds. Combining the statements of Lemmata~\ref{lemma:rep-W}, ~\ref{lemma:rep-V}, ~\ref{lemma:rep-l}, ~\ref{lemma:rep-L}, we get: for $t \le \ol{\tau}_{\CK}$, 
$$
\ol{Z}(t) = \ol{V}(t) + \ol{L}(t) = z_0 + \int_0^tg_0(\ol{Z}(s))\md s + \int_0^t\si_0(\ol{Z}(s))\md \ol{W}(s) + \int_0^tr_0(\ol{Z}(s))\md \ol{l}(s),
$$
where $\ol{W}$ behaves as a Brownian motion until $\ol{\tau}_{\CK}$, and the process $\ol{l}$ is continuous, nondecreasing, can increase only when $\ol{Z} \in \pa D$, and $\ol{l}(0) = 0$. Therefore, $\ol{Z}$ behaves as $Z_0$ until it exits $\Int\CK$. Apply Lemma~\ref{lemma:local} and finish the proof. 

\subsection{Proof of Lemma~\ref{lemma:tight-phi}} The sequence of the processes $(Z^{\CK}_n)_{n \ge 1}$ satisfy the following condition: for every $\de > 0$, 
$$
\lim\limits_{n \to \infty}\MP\left(\min\limits_{0 \le t \le T}\phi_0\left(Z^{\CK}_n(t)\right) \ge -\de\right) = 1.
$$
This is analogous to \cite[Lemma 4.2]{MyOwn8}, but here it is much easier to prove. 
Indeed, $Z^{\CK}_n(t) \in \ol{D}_n\cap\CK$. But we know that $D_n\Ra D_0$. From Lemma~\ref{lemma:equiv} (iii), we get:
$$
\varliminf\limits_{n \to \infty}\min\limits_{x \in \ol{D}_n\cap\CK}\phi_0(x) \ge 0.
$$
So there exists an $n_0(\de)$ such that for $n \ge n_0(\de)$ we have:
\begin{equation}
\label{eq:bound-for-phi}
\min\limits_{x \in \ol{D}_n\cap\CK}\phi_0(x) \ge -\de.
\end{equation}
Suppose that the following event happened:
\begin{equation}
\label{eq:event}
\left\{\oa\left(\phi_0\left(Z^{\CK}_n(\cdot)\right), [0, T], \eps\right) \ge 3\de\right\}.
\end{equation}
Then there exist $t_1, t_2 \in [0, T]$ such that $\phi_0\left(Z^{\CK}_n(t_1)\right) - \phi_0\left(Z^{\CK}_n(t_2)\right) \ge 3\de$ and $|t_1 - t_2| \le \eps$. Let 
$$
s_1 := t_1\wedge\tau_{\CK, n},\ \ s_2 := t_2\wedge\tau_{\CK, n}.
$$
Then $s_1, s_2 \in [0, \tau_{\CK, n}]$ and $|s_1 - s_2| \le \eps$. Also, 
$\phi_0(Z_n(s_1)) - \phi_0(Z_n(s_2)) \ge 3\de$. 
Now, $\phi_0(Z_n(s_2)) \ge -\de$ because of~\eqref{eq:bound-for-phi}. By continuity of 
$\phi_0(Z_n(\cdot))$, there exists $s_0$ between $s_1$ and $s_2$ such that 
$$
\phi_0(Z_n(s)) \ge \de\ \ \mbox{for}\ \ s\ \ \mbox{between}\ \ s_1, s_0,\ \ \mbox{and}\ \ \phi_0(Z_n(s_1)) - \phi_0(Z_n(s_0)) \ge \de.
$$
Certainly, $|s_0 - s_1| \le \eps$. But the function $\phi_0$ is $1$-Lipschitz, and so 
$$
\norm{Z_n(s_1) - Z_n(s_0)} \ge \phi_0(Z_n(s_1)) - \phi_0(Z_n(s_0)) \ge \de.
$$
For $s \in [0, \tau_{\CK, n}]$, we have: $Z_n(s) \in \CK$. Since $\phi_0(Z_n(s)) \ge \de$ for $s$ between $s_0$ and $s_1$, we have: $Z_n(s_1) - Z_n(s_0) = V_n(s_1) - V_n(s_0)$. Therefore,
\begin{equation}
\label{eq:diff-V}
\norm{V_n(s_1) - V_n(s_0)} = \norm{Z_n(s_1) - Z_n(s_0)} \ge \de.
\end{equation}
Taking $u_1 = s_1, u_2 = s_0$, we get from~\eqref{eq:diff-V} that the following event actually happened: 
\begin{equation}
\label{eq:contain}
\left\{\oa\left(\phi_0\left(Z^{\CK}_n(\cdot)\right), [0, T], \eps\right) \ge 3\de\right\} \subseteq A_n(\eps),
\end{equation}
where we define 
$$
A_n(\eps) := \left\{\exists\, u_1, u_2 \in [0, T]\mid |u_1 - u_2| \le \eps,\ \norm{V_n(u_1) - V_n(u_2)} \ge \de\right\}.
$$
Now, the sequence $(V_n(\cdot\wedge\tau_{\CK, n}))_{n \ge 1}$ is tight. Indeed, we can write 
$$
V_n\left(t\wedge\tau_{\CK, n}\right) = z_n + \int_0^tg_n(Z_n(s))1_{\{s \le \tau_{\CK, n}\}}\md s + \int_0^t\si_n(Z_n(s))1_{\{s \le \tau_{\CK, n}\}}\md W_n(s).
$$
Now, from Lemma~\ref{lemma:sup-g-sigma} below, there exists $n_1$ such that for $n \ge n_1$,
$$
\left|g_n(Z_n(s))\right| \le C_g,\ \ \left|\si_n(Z_n(s))\right| \le C_{\si},\ \ s \le \tau_{\CK, n}.
$$
Therefore, for all $s \in [0, T]$ and $n \ge n_1$, 
$$
\left|g_n(Z_n(s))1_{\{s \le \tau_{\CK, n}\}}\right| \le C_g,\ \ \left|\si_n(Z_n(s))1_{\{s \le \tau_{\CK, n}\}}\right| \le C_{\si}.
$$
By \cite[Lemma 7.4]{MyOwn6} (applied to the local martingale part) and the Arzela-Ascoli criterion (applied to the bounded variation part), the sequence $\left(V_n\left(\cdot\wedge\tau_{\CK, n}\right)\right)_{n \ge 1}$ is tight. Therefore, 
\begin{equation}
\label{eq:to-0}
\lim\limits_{\eps \to 0}\sup\limits_{n \ge 1}\MP\left(\exists\, u_1, u_2 \in [0, T] \mid 
|u_1 - u_2| \le \eps,\ \ \norm{V_n(u_1) - V_n(u_2)} \ge \de\right) = 0.
\end{equation}
Comparing~\eqref{eq:to-0} with~\eqref{eq:contain}, we get:
$$
\lim\limits_{\eps \to 0}\sup\limits_{n \ge 1}\MP\left(\oa\left(\phi_0\left(Z^{\CK}_n(\cdot)\right), [0, T], \eps\right) \ge 3\de\right) = 0.
$$
Apply the Arzela-Ascoli criterion and complete the proof. 

\begin{lemma} There exists an $n_0$ and constants $C_g, C_{\si}, C_r$ such that for $n \ge n_0$, we have: 
\begin{equation}
\label{eq:sup-g-sigma}
\sup\limits_{x \in \CK\cap \ol{D}_n}\norm{g_n(x)} \le C_g, \ \ 
\sup\limits_{x \in \CK\cap \ol{D}_n}\norm{\si_n(x)} \le C_{\si},\ \ \sup\limits_{x \in \CK\cap\pa D_n}\norm{r_n(x)} \le C_r.
\end{equation}
\label{lemma:sup-g-sigma}
\end{lemma}

\begin{proof} Let us prove this for $g_n$; the proofs for $\si_n$ and $r_n$ are similar. Assume the converse; then there exist $n_k \to \infty$ and $x_{n_k} \in \CK\cap\ol{D}_{n_k}$ such that $\norm{g_{n_k}(x_{n_k})} \to \infty$. But the set $\CK$ is compact, so there is a convergent subsequence $x_{n'_k} \to x_0 \in \CK\cap\ol{D}_0$.  
Therefore, $g_{n'_k}(x_{n'_k}) \to g_0(x_0)$. This contradiction completes the proof. 
\end{proof}

\subsection{Proof of Lemma~\ref{lemma:tight-V}} For all $s \ge 0$, $Z^{\CK}_n(s) \in \CK\cap\ol{D}_n$. We can conclude that the sequence 
$$
t \mapsto \int_0^{t\wedge\tau_{\CK, n}}g_n(Z^{\CK}_n(s))\md s
$$
is tight by Arzela-Ascoli criterion. Next, the sequence
$$
\ol{M}_n(t) := \int_0^{t\wedge\tau_{\CK, n}}\si_n(Z^{\CK}_n(s))\md W_n(s)
$$
is tight by \cite[Lemma 6.4]{MyOwn8}. Indeed, each $\ol{M}_n$ is a continuous local martingale with $\ol{M}_n(0) = 0$, and
$$
\langle \ol{M}_n\rangle_t = \int_0^{t\wedge\tau_{\CK, n}}\norm{\si_n(Z^{\CK}_n(s))}^2\md s.
$$
But $Z^{\CK}_n(s) \in D_n\cap\CK$ for all $s \in [0, T]$. Apply Lemma~\ref{lemma:sup-g-sigma} and complete the proof.

\subsection{Proof of Lemma~\ref{lemma:tight-l}} Let us state a technical lemma, which is proved in Appendix. 

\begin{lemma} For every compact subset $\CK \subseteq \BR^d\setminus\CV_0$, there exists a $\de_{\CK} \in (0, \dist(\CK, \CV_0)/2)$ such that:

\smallskip

(i) the signed distance function $\phi_0$ is $C^2$ on the set 
\begin{equation}
\label{eq:CK0}
\CK' := \{x \in \CK\mid |\phi_0(x)| \equiv \dist(x, \pa D_0) \le \de_{\CK}\};
\end{equation}

\smallskip

(ii) for every $x \in \CK'$, there exists a unique point $\zeta(x) \in \pa D_0\setminus\CV_0$  which is the closest to $x$ on $\pa D_0$: $\norm{x - \zeta(x)} = \dist(x, \pa D_0) = \dist(x, \pa D_0\setminus\CV_0)$, and this function $\zeta$ is continuous on $\CK'$. 
\label{lemma:aux}
\end{lemma}

Take a $C^{\infty}$ function $\psi : \BR \to \BR$ such that 
$$
\psi(x) := 
\begin{cases}
x,\ |x| \le \de_{\CK}/2;\\
0,\ |x| \ge \de_{\CK}.
\end{cases}
$$
Let us write an It\^o equation for the process $\psi(\phi_0(Z^{\CK}_n(\cdot)))$, or, equivalently, for $(\psi\circ\phi_0)(Z_n(t))$ for $t \le \tau_{\CK, n}$. We have: $\psi\circ\phi_0 \in C^2$ on $\CK$. Therefore, we can apply It\^o formula for the function $\psi\circ\phi_0$. We have: $\nabla(\psi\circ\phi_0)(x) = \psi'(\phi_0(x))\nabla\phi_0(x)$. 
Abusing the notation, we can write this even if $|\phi_0(x)| > \de_{\CK}$, where the function $\phi_0$ might not be $C^2$, since then $\psi'(\phi_0(x)) = 0$ and the left-hand side is also zero. In addition, a similar formula holds for second derivatives:
$$
\theta_{ij}(x) := \frac{\pa^2(\psi\circ\phi_0)(x)}{\pa x_i\pa x_j} = \psi''(\phi_0(x))\frac{\pa\phi_0}{\pa x_i}\frac{\pa\phi_0}{\pa x_j} + \psi'(\phi_0(x))\frac{\pa^2\phi_0}{\pa x_i\pa x_j}.
$$
By It\^o's formula, for $t \le \tau_{\CK, n}$, 
\begin{equation}
\label{eq:ito}
\md \psi(\phi_0(Z_n(t))) = \psi'(\phi_0(Z_n(t)))\nabla\phi_0(Z_n(t))\cdot \md Z_n(t) + 
\SL_{i=1}^d\SL_{j=1}^d\theta_{ij}(Z_n(t))\md\langle (Z_n)_i, (Z_n)_j\rangle_t.
\end{equation}
Now, from~\eqref{eq:Z=V+L} and the fact that $L_n$ has finite variation, we get: for $t \le \tau_{\CK, n}$, 
$$
\md\langle (Z_n)_i, (Z_n)_j\rangle_t = (\si_n\si^T_n)_{ij}(Z_n(t))\md t.
$$
From the properties of $\phi_0$ and $\psi$ it follows that the function $\psi'(\phi_0(x))\nabla\phi_0(x)$, as well as each $\theta_{ij}$ is bounded on $\CK$. Apply Lemma~\ref{lemma:sup-g-sigma} and note that $Z_n(t) \in \ol{D}_n\cap \CK$ for $t \le \tau_{\CK, n}$. By the Arzela-Ascoli criterion, the following sequence  is tight:
$$
t \mapsto \int_0^{t\wedge\tau_{\CK, n}}\SL_{i=1}^d\SL_{j=1}^d\theta_{ij}(Z_n(t))\md\langle (Z_n)_i, (Z_n)_j\rangle_t
$$
Take the first term in the right-hand side of~\eqref{eq:ito}
\begin{align*}
\psi'(\phi_0(Z_n(t)))\nabla&\phi_0(Z_n(t))\cdot \md Z_n(t) = \psi'(\phi_0(Z_n(t)))\nabla\phi_0(Z_n(t))\cdot g_n(Z_n(t))\md t \\ & + \psi'(\phi_0(Z_n(t)))\nabla\phi_0(Z_n(t))\cdot \si_n(Z_n(t))\md W_n(t) \\ &\ \ \ \ \ + \psi'(\phi_0(Z_n(t)))\nabla\phi_0(Z_n(t))\cdot r_n(Z_n(t))\md l_n(t).
\end{align*}
By Lemma~\ref{lemma:sup-g-sigma} and the Arzela-Ascoli criterion, the following sequence is tight:
$$
t \mapsto \int_0^{t\wedge\tau_{\CK, n}}\psi'(\phi_0(Z_n(s)))\nabla\phi_0(Z_n(s))\cdot g_n(Z_n(s))\md s
$$
Next, the following sequence of continuous local martingales
$$
M_n(t) := \int_0^{t\wedge\tau_{\CK, n}}\psi'(\phi_0(Z_n(s)))\nabla\phi_0(Z_n(s))\cdot \si_n(Z_n(s))\md W_n(s)
$$
is tight by Lemma 6.4 from \cite{MyOwn8}. Indeed, 
$$
\langle M_n\rangle_t = \int_0^{t\wedge\tau_{\CK, n}}\psi'^2(\phi_0(Z_n(s)))\norm{\nabla\phi_0(Z_n(s))\cdot \si_n(Z_n(s))}^2\md s,
$$
and the derivative of this function with respect to $t$ is uniformly bounded. (This follows from the fact that $\psi'$ is bounded on $\BR$, $\phi_0$ is bounded on $\CK$, and from Lemma~\ref{lemma:sup-g-sigma}. By Lemma 6.7 from the same article \cite{MyOwn8}, the sequence $\psi(\phi_0(Z_n^{\CK}(\cdot)))$ is itself tight. Therefore, the sequence 
$$
N_n(t) := \int_0^{t\wedge\tau_{\CK, n}}\psi'(\phi_0(Z_n(s)))\nabla\phi_0(Z_n(s))\cdot r_n(Z_n(s))\md l_n(s)
$$
is tight in $C([0, T], \BR^d)$. But the process $l_n$ can grow only when $Z_n \in \pa D_0$, that is, when $\phi_0(Z_n(s)) = 0$. For these $s$ we have: $\psi'(\phi_0(Z_n(s))) = 1$, because $\psi'(0) = 1$. Therefore, we can rewrite 
\begin{equation}
\label{eq:N-n}
N_n(t) := \int_0^{t\wedge\tau_{\CK, n}}\nabla\phi_0(Z_n(s))\cdot r_n(Z_n(s))\md l_n(s).
\end{equation}

\begin{lemma} There exists $n_0$ and $\eps_0 > 0$ such that for $n \ge n_0$, for $x \in \pa\ol{D}_n\cap\CK$, we have: $\nabla\phi_0(x)\cdot r_n(x) \ge \eps_0$. 
\label{lemma:eps-0}
\end{lemma}

\begin{proof} Assume the converse. Then there exist a subsequence $(n_k)_{k \ge 1}$ and a corresponding sequence of points $x_{n_k} \in \pa\ol{D}_{n_k}\cap \CK$ such that 
$$
\nabla\phi_0\left(x_{n_k}\right)\cdot r_{n_k}\left(x_{n_k}\right) \le \frac1k.
$$
Since $\CK$ is compact, there exists a subsequence $(n'_k)_{k \ge 1}$ such that $x_{n'_k} \to x_0$. Then $x_0 \in \CK\cap\pa D_0$. For all $k \ge k_0$, $x_{n'_k} \in \CK'$ (and $x_0 \in \CK'$). But $\nabla\phi_0$ is continuous on $\CK'$. Therefore, 
$\nabla\phi_0\left(x_{n_k}\right) \to \nabla\phi_0(x_0)$. 
Also, since $r_n \Ra r_0$, we have: $r_{n_k}\bigl(x_{n_k}\bigr) \to r_0(x_0)$. 
Therefore, passing to the limit, we have: $\nabla\phi_0(x_0)\cdot r_0(x_0) \le 0$. 
But $\nabla\phi_0(x_0)$ has the same direction as the inward unit normal vector $\fn(x_0)$ to $\pa D_0$, and by the properties of the reflection field $r_0$ we have:
$\fn(x_0)\cdot r_0(x_0) > 0$. This contradiction completes the proof. 
\end{proof}

In view of Lemma~\ref{lemma:eps-0}, we can rewrite~\eqref{eq:N-n} as
$$
l_n\left(t\wedge\tau_{\CK, n}\right) := \int_0^{t\wedge\tau_{\CK, n}}\left[\nabla\phi_0(Z_n(s))\cdot r_n(Z_n(s))\right]^{-1}\md N_n(t). 
$$
But $(N_n)_{n \ge 1}$ is tight, and by Lemma~\ref{lemma:eps-0} we have:
$$
\left[\nabla\phi_0(Z_n(s))\cdot r_n(Z_n(s))\right]^{-1} \le \eps_0^{-1}.
$$
Therefore, $l_n\left(\cdot\wedge\tau_{\CK, n}\right)$ is tight. The proof is complete.  

\subsection{Proof of Lemma~\ref{lemma:tight-L}} Note that the process $l_n$ can grow only when $Z_n \in \pa D_n$. By Lemma~\ref{lemma:sup-g-sigma}, for $n \ge n_0$, 
$$
\sup\limits_{0 \le s \le t\wedge\tau_{\CK, n}}\norm{r_n(Z_n(s))} \le C_r.
$$
Therefore, the sequence $(L_n)_{n \ge 1}$ is also tight. 

\subsection{Proof of Lemma~\ref{lemma:rep-V}} Without loss of generality, assume $n_k = k$ for convenience of notation. We have: 
$Z_{k}^{\CK} \to \ol{Z}$ uniformly on $[0, T]$, and 
$$
Z_{k}^{\CK}(s) \in \ol{D}_{k}\cap\CK,\ \ \mbox{and}\ \ \ol{Z}(s) \in \ol{D}_0\cap\CK\ \ \mbox{for}\ \ s \in [0, T].
$$
Recall the definition of locally uniform convergence of functions defined on different subsets of $\BR^d$. Since $g_n \Ra g_0$, $\si_n \Ra \si_0$ by Remark~\ref{rmk:sqrt}, and 
$$
Z^{\CK}_{k}(s) \to  \ol{Z}(s) \ \ \mbox{uniformly on}\ \ [0, T],
$$
by Lemma~\ref{lemma:equiv} (v) we have:
\begin{equation}
\label{eq:conv-chain}
g_{k}(Z^{\CK}_{k}(s)) \to g_0(\ol{Z}(s)),\ \ \si_{k}(Z^{\CK}_{k}(s)) \to \si_0(\ol{Z}(s)).
\end{equation}
From Lemma~\ref{lemma:conv-of-stochastic-integrals}, we have:
\begin{equation}
\label{eq:conv-of-si}
\int_0^{t\wedge\tau_{\CK, k}}\si_k(Z^{\CK}_k(s))\md W_k(s) = 
\int_0^t\si_k(Z^{\CK}_k(s))\md W_k^{\CK} \to 
\int_0^t\si_0(\ol{Z}(s))\md\ol{W}(s),
\end{equation}
where the convergence is understood in probability. Therefore, there exists a subsequence $(k_m)_{m \ge 1}$ such that 
\begin{equation}
\label{eq:conv-of-si-a.s.}
\int_0^{t\wedge\tau_{\CK, k_m}}\si_{k_m}(Z^{\CK}_{k_m}(s))\md W_{k_m}(s)  \to \int_0^t\si_0(\ol{Z}(s))\md\ol{W}(s)\ \ \mbox{a.s. uniformly on}\ \ [0, T].
\end{equation}

\begin{lemma} Uniformly on $[0, \ol{\tau}_{\CK}]$, we have: 
$$
\int_0^{t\wedge\tau_{\CK, n}}g_k(Z^{\CK}_k(s))\md s \to \int_0^{t\wedge\ol{\tau}_{\CK}}g_0(\ol{Z}(s))\md s.
$$
\label{lemma:conv-of-g}
\end{lemma}

\begin{proof} For every $\eps > 0$ there exists $k_1(\eps)$ such that for $k \ge k_1(\eps)$ we have: $\tau_{\CK, k} \le \ol{\tau}_{\CK} + \eps$, and so for $t \le \ol{\tau}_{\CK}$ we have: $\left|t\wedge\ol{\tau}_{\CK} - t\wedge\tau_{\CK, k}\right| \le \eps$. Therefore, 
$$
\left|\int_0^{t\wedge\ol{\tau}_{\CK}}g_0(\ol{Z}(s))\md s - \int_0^{t\wedge\tau_{\CK, k}}g_0(\ol{Z}(s))\md s\right| \le \eps\cdot\max\limits_{[0, T]}\left|g_0(\ol{Z}(s))\right|. 
$$
From~\eqref{eq:conv-chain}, we have: $g_k(Z^{\CK}_k(t)) \to g_0(\ol{Z}(t))$  uniformly on $[0, T]$. Therefore, there exists $n_{\eps}$ such that for $n \ge n_{\eps}$ we have: 
\begin{equation}
\label{eq:430}
\max\limits_{t \in [0, T]}\norm{g_n(Z^{\CK}_n(t)) - g_0(\ol{Z}(t))} \le \eps.
\end{equation}
We have: for $n \ge n_{\eps}$, 
\begin{equation}
\label{eq:431}
\left|\int_0^{t\wedge\tau_{\CK, n}}g_n(Z^{\CK}_n(s))\md s - \int_0^{t\wedge\tau_{\CK, n}}g_0(\ol{Z}(s))\md s\right| \le T\cdot\max\limits_{t \in [0, T]}\norm{g_n(Z^{\CK}(t)) - g_0(\ol{Z}(t))} \le T\eps.
\end{equation}
Combining~\eqref{eq:430} and~\eqref{eq:431}, we have: for $n \ge n_{\eps}$, 
$$
\left|\int_0^{t\wedge\ol{\tau}_{\CK}}g_0(\ol{Z}(s))\md s - \int_0^{t\wedge\tau_{\CK, n}}g_n(Z^{\CK}_n(s))\md s\right| \le \eps\cdot\left(T + \max\limits_{[0, T]}\left|g_0(\ol{Z}(s))\right|\right). 
$$
Since $\eps > 0$ is arbitrary, the proof is complete. 
\end{proof}

Combining Lemma~\ref{lemma:conv-of-g} with~\eqref{eq:conv-of-si-a.s.} and $z_n \to z_0$, we get: uniformly on $[0, \ol{\tau}_{\CK}]$, 
\begin{align*}
V_{k_m}(t) &= z_{k_m} + \int_0^{t\wedge\tau_{\CK, k_m}}g_{k_m}(Z^{\CK}_{k_m}(s))\md s + \int_0^{t\wedge\tau_{\CK, k_m}}\si_{k_m}(Z^{\CK}_{k_m}(s))\md W_{k_m}(s) \\ & \ \ \ \to \ \ z_0 + \int_0^{t\wedge\ol{\tau}_{\CK}}g_0(\ol{Z}(s))\md s + \int_0^{t\wedge\ol{\tau}_{\CK}}\si_0(\ol{Z}(s))\md W_0(s).
\end{align*}
But $V_n \to V_0$ a.s. uniformly on $[0, T]$. This completes the proof of Lemma~\ref{lemma:rep-V}.

\begin{lemma}
\label{lemma:conv-of-stochastic-integrals}
For $m \ge 1$, let $Y_k = (Y_k(t), 0 \le t \le T)$ be an $\BR^d$-valued continuous adapted process, and let $U_k = (U_k(t), 0 \le t \le T)$ be an $\BR^d$-valued continuous local martingale. If in $C([0, T], \BR^d\times\BR^d)$ we have:
$(Y_k, U_k) \Ra (Y, U)$, $k \to \infty$, then $U$ is a semimartingale, and we have the following convergence in probability:
$$
\int_0^{t}Y_k\md U_k \to \int_0^{t}Y\md U\ \ \mbox{uniformly on}\ \ t \in [0, T].
$$
\end{lemma}
 
This lemma was proved in \cite[Theorem 5.10]{KurtzProtter1991}; see also \cite[Lemma 3.6]{Thesis}. Both of these statements are more general than Lemma~\ref{lemma:conv-of-stochastic-integrals}. For convenience, we state this result here in the form which is convenient for our use. 

\subsection{Proof of Lemma~\ref{lemma:rep-l}} As before, assume for simplicity that $n_k = k$. Fix $\eps > 0$ and let us prove these properties for $\ol{l}$ on $[0, \ol{\tau}_{\CK} - \eps]$. Note that there exists $n(\eps)$ such that for $k \ge n(\eps)$ we have: $\ol{\tau}_{\CK} \le \tau_{\CK, k} + \eps$. 
Now, $l_k^{\CK} \to \ol{l}(t)$ uniformly on $[0, T]$; but $l_k^{\CK}(t) \equiv l_k\left(t\wedge\tau_{\CK, k}\right) \equiv l_k(t)$ for $t \in [0, \tau_{\CK, k}] \subseteq [0, \ol{\tau}_{\CK} - \eps]$. Now, $l_k$ is nondecreasing and $l_k(0) = 0$; therefore, the same properties hold for $\ol{l}$ on $[0, \ol{\tau}_{\CK} - \eps]$. 

Fix $\de > 0$ and let us show that $\ol{l}$ does not increase on $[t_1, t_2] \subseteq [0, \ol{\tau}_{\CK} - \eps]$ if $\dist(\ol{Z}(t), \pa D_0) > \de$ for $t \in [t_1, t_2]$. Indeed, since $Z^{\CK}_k \to \ol{Z}$ uniformly on $[0, T]$, by Lemma~\ref{lemma:equiv} (v) we have:
$$
\phi_k(Z_k) \to \phi_0(\ol{Z})\ \ \mbox{uniformly on}\ \ [t_1, t_2].
$$
But $\dist(Z_k(t), \pa D_k) \equiv |\phi_k(Z_k(t))|$. Therefore, 
$$
\dist(Z_k(t), \pa D_k) \to \dist(\ol{Z}(t), \pa D_0)\ \ \mbox{uniformly on}\ \ [t_1, t_2].
$$
Therefore, for $k \ge m(\de)$, $t \in [t_1, t_2]$, we have: $\dist(Z_k(t), \pa D_k) \ge \de/2$. Meanwhile, $l_k$ does not grow on $[t_1, t_2]$: that is, $l_k(t_1) = l_k(t_2)$. Let $k \to \infty$ and conclude: $\ol{l}(t_1) = \ol{l}(t_2)$. Thus, $\ol{l}$ does not grow on $[t_1, t_2]$. 

Now, let us prove a more general statement: if $[t_1, t_2] \subseteq [0, \ol{\tau}_{\CK}]$ and $\dist(\ol{Z}(t), \pa D_0) > 0$ for $t \in [t_1, t_2]$, then $\ol{l}(t_1) = \ol{l}(t_2)$. Indeed, assume $\ol{l}(t_1) < \ol{l}(t_2)$. By continuity of $\ol{l}$, there exists $\eps > 0$ such that $\ol{l}(t_1) < \ol{l}(t_2-\eps)$. By continuity of $\ol{Z}$, there exists $\de > 0$ such that $\dist(\ol{Z}(t), \pa D_0) \ge \de$ for $t \in [t_1, t_2]$. Now, repeat the previous argument and conclude: $\ol{l}(t_1) = \ol{l}(t_2 - \eps)$. This contradiction completes the proof.

\subsection{Proof of Lemma~\ref{lemma:rep-L}} As before, we assume $n_k = k$ without loss of generality. There exists $n_0$ such that for $n \ge n_0$, we have: for $x \in \pa D_n\cap\CK$, $|\phi_0(x)| \le \de_{\CK}$. This follows from Lemma~\ref{lemma:equiv}. In other words, for $n \ge n_0$ we have: $\pa D_n\cap \CK \subseteq \CK'$, where $\CK'$ was defined in~\eqref{eq:CK0}. By Lemma~\ref{lemma:aux} (ii), the distance function $\zeta$ is continuous on $\CK'$. Note that 
$$
\eps_n := \max\limits_{x \in \CK\cap\pa D_n}|\phi_0(x)| \to 0.
$$
For $x \in \CK_0$, we have: $\norm{\zeta(x) - x} = \dist(x, \pa D_0) = |\phi_0(x)| \le \eps_n$. Therefore, by definition of locally uniform convergence $r_n \Ra r_0$, we have:
\begin{equation}
\label{eq:sup-10}
\sup\limits_{x \in \CK\cap\pa D_n}\norm{r_n(x) - r_0(\zeta(x))} \to 0.
\end{equation}
Therefore, we get:
$$
\int_0^tr_k(Z^{\CK}_k(s))\md l^{\CK}_k(s) - \int_0^tr_0(\ol{Z}(s))\md \ol{l}(s) = I_1(k) +I_2(k) + I_3(k),
$$
$$
I_1(k) := \int_0^t\left[r_k(Z^{\CK}_k(s)) - r_0(\zeta(Z^{\CK}_k(s)))\right]\md l^{\CK}_k(s),
$$
$$
I_2(k) := \int_0^t\left[r_0(\zeta(Z^{\CK}_k(s))) - r_0(\ol{Z}(s))\right]\md \ol{l}(s),
$$
$$
I_3(k) := \int_0^tr_0(\zeta(Z^{\CK}_k(s)))\md l^{\CK}_k(s) - \int_0^tr_0(\zeta(Z^{\CK}_k(s)))\md \ol{l}(s).
$$
By Lemma 6.2 from \cite{MyOwn8}, $\norm{I_1(k)} \to 0$ as $k \to \infty$. 
From the relation~\eqref{eq:sup-10}, the fact that each $l^{\CK}_k$ is nondecreasing, and the convergence $l^{\CK}_k(T) \to \ol{l}(T)$, we have: $\norm{I_3(k)} \to 0$. 
Finally, $\zeta(\ol{Z}(s)) = \ol{Z}(s)$ when $\ol{Z}(s) \in \pa D_0$. But the function $\ol{l}$ can grow only when $\ol{Z}(s) \in \pa D_0$: this follows from Lemma~\ref{lemma:rep-l}. Therefore,
$$
\int_0^tr_0(\ol{Z}(s))\md\ol{l}(s) = \int_0^tr_0(\zeta(\ol{Z}(s)))\md\ol{l}(s).
$$
The function $\zeta$ is continuous on $\CK_0$, and since $Z_k^{\CK} \to \ol{Z}$ uniformly on $[0, T]$, we have: $\zeta(Z_k^{\CK}) \to \zeta(\ol{Z})$ uniformly on $[0, T]$. Therefore, as $k \to \infty$, 
$$
\norm{I_2(k)} \le \max\limits_{0 \le s \le T}\norm{r_0(\ol{Z}(s)) - r_0(\zeta(Z^{\CK}_k(s)))}\cdot\ol{l}(T) \to 0.
$$

\section{Convergence to a Non-Reflected Diffusion} 

\subsection{Convergence of domains to the whole space} 

Let us modify the definition of weak convergence $D_n \Ra D_0$ for the case $D_0 = \BR^d$. The main question is how to define $\phi_0(x)$, the signed distance from $x \in \BR^d$ to the boundary $\pa D_0$, because the set $D_0 = \BR^d$ has no boundary: $\pa\BR^d = \varnothing$. Intuitively, we can approximate $\BR^d$ by a very large ball $U(0, r)$. Take a point $x \in \BR^d$. Since $r$ is large, $x \in U(0, r)$, and the distance from $x$ to $\pa U(0, r)$ is equal to $r - \norm{x}$, which is also large. Therefore, it makes sense to define $\phi_0(x) := \infty$ for all $x \in \BR^d$. 

\begin{defn}
We say that a sequence of domains $(D_n)_{n \ge 1}$ {\it converges weakly to $\BR^d$} and write $D_n \Ra \BR^d$, if $\phi_n(x) \to \infty$ for all $x \in \BR^d$. 
\end{defn}

The following is an equivalent characterization of this weak convergence. The proof is postponed until Appendix.

\begin{lemma} $D_n \Ra \BR^d$ if and only if for every compact set $\CK \subseteq \BR^d$ there exists an $n_0$ such that for $n > n_0$ we have: $\CK \subseteq D_n$.
\label{lemma:conv-to-Rd}
\end{lemma}

Let us state an analogue of Theorem~\ref{thm:main} for the case when $D_n \Ra \BR^d$, In this case, reflected diffusions $Z_n$ converge weakly to a {\it non-reflected} diffusion $Z_0$ in $\BR^d$. Take a sequence $(D_n)_{n \ge 1}$ of domains in $\BR^d$. For each $n \ge 1$, consider a reflected diffusion $Z_n = (Z_n(t), t \ge 0)$ in $\ol{D}_n$ with drift vector $g_n(\cdot)$, covariance matrix $A_n(\cdot)$, and reflection field $r_n(\cdot)$, starting from $Z_n(0) = z_n$. We suppose that Assumptions 1, 2, 3 are satisfied. We do {\it not} impose a condition that $Z_n$ does not hit non-smooth parts 
$\CV_n$ of the boundary $\pa D_n$. Define a drift coefficient $g_0 : \BR^d \to \BR^d$ and a covariance matrix $A_0 : \BR^d \to \CP_d$. For each $x \in \BR^d$, let $\si_0(x) = A^{1/2}(x)$. Consider a {\it non-reflected} diffusion process
$$
\md Z_0(t) = g_0(Z_0(t))\md t + \si_0(Z_0(t))\md W(t),\ \ Z_0(0) = z_0.
$$
Assume it exists and is unique in the weak sense. 

\begin{thm}
\label{thm:Rd}
Assume $D_n \Ra \BR^d$ weakly, and $g_n \Ra g_0$, $A_n \Ra A_0$, $z_n \to z_0$. 
Then $Z_n \Ra Z_0$ in $C([0, T], \BR^d)$ for every $T > 0$. 
\end{thm}

\begin{proof} We modify the proof of Theorem~\ref{thm:main} a bit. First, fix a compact set $\CK \subseteq \BR^d$ such that $z_0 \in \Int\CK$. It suffices to show that $Z_n^{\CK} \Ra Z_0^{\CK}$, then apply Lemma~\ref{lemma:local}. (It is stated and proved for a non-reflected $Z_0$ in the same way as for the case of a reflected diffusion $Z_0$.) By Lemma~\ref{lemma:conv-to-Rd}, there exists $n_0$ such that $\CK \subseteq D_n$ for $n > n_0$. So $L_n(t) \equiv 0$, and $Z_n^{\CK} \equiv V_n$ (we use the notation from the proof of Theorem~\ref{thm:main}). The rest of the proof is reduced to Lemmata~\ref{lemma:tight-V} and~\ref{lemma:rep-V}. 
\end{proof}

\subsection{Convergence of domains to ``almost'' the whole space} Now, assume $D_n \Ra D_0 = \BR^d\setminus\CM$, where $\CM \subseteq \BR^d$ is a ``set of dimension'' less than or equal to $d-2$. Then the limiting diffusion $Z_0$ (under some conditions) does not hit $\CM$, so this is actually a {\it non-reflected diffusion}. We use the notation of the previous subsection. We again suppose that Assumptions 1, 2, 3 are satisfied, and we do {\it not} impose a condition that $Z_n$ does not hit non-smooth parts $\CV_n$ of the boundary $\pa D_n$.

\begin{thm}
\label{thm:Rd-almost} In the notation of the previous subsection, assume 
$$
D_n \Ra D_0 = \BR^d\setminus\CM,\ \ g_n \Ra g_0,\ \ A_n \Ra A_0,\ \ z_n \to z_0.
$$
Finally, assume that the diffusion $Z_0 = (Z_0(t), t \ge 0)$, defined by
$$
\md Z_0(t) = g_0(Z_0(t))\md t + \si_0(Z_0(t))\md W(t),\ \ Z_0(0) = z_0, 
$$
a.s. does not hit the set $\CM$:
$$
\MP\left(\exists t \ge 0:\ Z_0(t) \in \CM\right) = 0.
$$
Then $Z_n \Ra Z_0$ in $C([0, T], \BR^d)$. 
\end{thm}

\begin{rmk} Sufficient conditions for $Z_0$ not hitting $\CM$, when $\CM$ is a submanifold in $\BR^d$ of dimension less than or equal to $d - 2$, can be found in 
\cite{Ramasubramanian1988, Ramasubramanian1983}. 
\end{rmk}

\begin{proof} As in the proof of Theorem~\ref{thm:Rd}, we follow the proof of Theorem~\ref{thm:main}. Fix any compact set $\CK \subseteq D_0$. It suffices to prove that $Z_n^{\CK} \Ra Z_0^{\CK}$. By Corollary~\ref{cor:simple}, there exists $n_0$ such that for $n > n_0$, we have: $\CK \subseteq D_n$. Now, we just need to repeat the rest of the proof of Theorem~\ref{thm:Rd}. 
\end{proof}

\section{Appendix}

\subsection{Proof of Lemma~\ref{lemma:Wijsman}} (i) Similar to the proof of Lemma~\ref{lemma:equiv} below. 

\medskip

(ii) Fix $\eps > 0$ and let us show that $\dist(x_0, E_0) < 2\eps$. There exists $n_1$ such that for $n \ge n_1$ we have: $\norm{x_n - x_0} < \eps$. The set $\CK := \{x_n\mid n \ge 1\}$ is compact; therefore, $\dist(x, E_n) \to \dist(x, E_0)$ uniformly on $\CK$, and there exists $n_2$ such that for $n \ge n_2$, we have: $\left|\dist(x, E_n) - \dist(x, E_0)\right| < \eps$ for $x \in \CK$. Take $n = n_1\vee n_2$. Then 
$$
\dist(x_0, E_0) \le \dist(x_0, E_n) + \norm{x_n - x_0} \le \dist(x_n, E_n) + \eps + \norm{x_n - x_0} \le 2\eps.
$$
Since $\eps > 0$ is arbitrary, $\dist(x_0, E_0) = 0$; therefore, $x_0 \in \ol{E}_0$. 

\subsection{Proof of Lemma~\ref{lemma:equiv}} {\bf (i) $\Ra$ (iii)} Assume the converse: there exists a compact subset $\CK \subseteq \BR^d$, a positive number $\eps > 0$ and a sequence $x_{n_k} \in \CK$ such that 
$$
\bigl|\phi_{n_k}\bigl(x_{n_k}\bigr)  - \phi_0\bigl(x_{n_k}\bigr)\bigr| \ge \eps.
$$
By compactness, there exists a limit point $x_0 := \lim x_{n'_k}$. There exists $k_0$ such that for $k \ge k_0$, we have: $\norm{x_{n'_k} - x_0} \le \eps/3$. But the signed distance functions $\phi_0$ and $\phi_{n'_k}$ are $1$-Lipschitz, see \cite{Delfour1994}. 
Therefore, for $k \ge k_0$ we get:
$$
\bigl|\phi_{n'_k}\bigl(x_{n'_k}\bigr) - \phi_{n'_k}(x_0)\bigr| \le \frac{\eps}3,
 \ \ \bigl|\phi_0\bigl(x_{n'_k}\bigr) - \phi_0(x_0)\bigr| \le \frac{\eps}3.
$$
Thus, for $k \ge k_0$ we have:
$$
\bigl|\phi_{n'_k}(x_0) - \phi_0(x_0)\bigr| \ge \eps - \frac{\eps}3 - \frac{\eps}3 = \frac{\eps}3.
$$
This contradicts the condition (i). 

\medskip

{\bf (i) $\Lra$ (ii)} Note that $\dist(x, \ol{D}_n) \equiv (\phi_n(x))_-$ and $\dist(x, D_n^c) \equiv (\phi_n(x))_+$. A sequence $(a_n)_{n \ge 1}$ of real numbers converges to $a_0$ if and only if $(a_n)_+ \to (a_0)_+$ and $(a_n)_- \to (a_0)_-$. The ``only if'' part follows from the fact that 
$x \mapsto x_+$ and $x \mapsto x_-$ are continuous functions; the ``if'' part follows from the fact that $x = x_+ - x_-$. The rest is trivial. 

\medskip

{\bf (ii) $\Ra$ (iv)} Since the function $\phi_0$ is $1$-Lipschitz, as proved in \cite{Delfour1994}, we have:
$$
\max\limits_{\substack{x_n, x_0 \in \CK\\ \norm{x_n - x_0} \le \eps_n}} \left|
\phi_n(x_n) - \phi_0(x_0)\right| \le \max\limits_{\substack{x_n, x_0 \in \CK\\ \norm{x_n - x_0} \le \eps_n}} \left|\phi_n(x_n) - \phi_0(x_n)\right| + \eps_n \to 0.
$$

\medskip

{\bf (iv) $\Ra$ (v)} Take $\CK = \{z \in \BR^d\mid \norm{z} \le \max\norm{f_0} + 1\}$ and $\eps_n := \max\norm{f_n(t) - f_0(t)}$ for $n = 1, 2, \ldots$

\medskip

{\bf (v) $\Ra$ (i)} Take constant functions $f_n(t) \equiv x$ for $t \in [0, T]$ and $n = 0, 1, 2, \ldots$

\medskip

{\bf (i) $\Ra$ (vi)} Note that $\dist(x, \pa D_n) \equiv |\phi_n(x)|$ for $x \in \BR^d$ and $n = 0, 1, 2, \ldots$ Since $\phi_n(x) \to \phi_0(x)$, we have: $|\phi_n(x)| \to |\phi_0(x)|$. Now, let us show~\eqref{eq:inclusion}. Take $x \in D_0$. Then $\lim\phi_n(x) = \phi_0(x) > 0$, and so there exists $n_x$ such that for $n \ge n_x$ we get:
$\phi_n(x) > 0$, and therefore $x \in D_n$. Thus, $x \in \varliminf D_n$. We conclude that $D_0 \subseteq \varliminf D_n$. Similarly, we can prove that $\ol{D}_0^c \subseteq \varliminf\limits_{n \to \infty}\ol{D}_n^c$, which is equivalent to $\varlimsup\limits_{n \to \infty}\ol{D}_n \subseteq \ol{D}_0$. 

\medskip

{\bf (vi) $\Ra$ (i)} We have: $|\phi_n(x)| \to |\phi_0(x)|$ for every $x \in \BR^d$, as $n \to \infty$. Consider three cases:

\medskip

{\it Case 1:} $\phi_0(x) > 0$, which is equivalent to $x \in D_0$. Using the first inclusion from~\eqref{eq:inclusion}, we get: $x \in \varliminf D_n$, and so there exists $n_x$ such that for $n \ge n_x$, we have: $x \in D_n$, and $\phi_n(x) > 0$. Therefore, 
$\phi_n(x) = |\phi_n(x)| \to |\phi_0(x)| = \phi_0(x)$. 

\medskip

{\it Case 2:} $\phi_0(x) < 0$, which is equivalent to $x \in \ol{D}_0^c$. Then we use the second inclusion from~\eqref{eq:inclusion} and complete the proof similarly to Case 1. 

\medskip

{\it Case 3:} $\phi_0(x) = 0$. Then $|\phi_n(x)| \to |\phi_0(x)| = 0$, and so $\phi_n(x) \to 0$. 

\medskip

{\bf (vi) $\Ra$ (vii)} Follows from Lemma~\ref{lemma:Wijsman} (ii) above. 

\medskip

{\bf (vii) $\Ra$ (vi)} Fix $x \in \BR^d$ and let us show that $\dist(x, \pa D_n) \to \dist(x, \pa D_0)$. 

\begin{lemma} Assume~\eqref{eq:inclusion} holds. Take $x_0 \in \pa D_0$. Then there exists a sequence $(x_n)_{n \ge 1}$ such that $x_n \in \pa D_n$ and $x_n \to x_0$. 
\label{lemma:tending}
\end{lemma}

\begin{proof} Assume the converse: there exists a neighborhood $U(x_0, \eps)$ and a subsequence $(n_k)_{k \ge 1}$ such that for $k \ge 1$, we have: $U(x_0, \eps)\cap\pa D_{n_k} = \varnothing$. Since $x_0 \in \pa D_0$, there exists $y \in U(x_0, \eps)\cap D_0$ and $z \in U(x_0, \eps)\cap\ol{D}_0^c$. Then $y \in \varliminf D_n$; that is, $y \in D_n$ for $n > n_y$; and $z \in \varliminf\ol{D}_n^c$, that is, $z \in \ol{D}_n^c$ for $n > n_z$. Let $k_0$ be large enough so that for $k \ge k_0$, $n_k > n_y\vee n_z$. Then $y \in D_{n_k}$ and $z \in \ol{D}_{n_k}^c$ for $k \ge k_0$. Therefore, $[y, z]\cap\pa D_{n_k} \ne \varnothing$; take $w_k \in [y, z]\cap\pa D_{n_k}$. But $[y, z] \subseteq U(x_0, \eps)$, because the open ball $U(x_0, \eps)$ is convex. Therefore, $U(x_0, \eps)\cap\pa D_{n_k} \ne \varnothing$. This contradiction completes the proof.
\end{proof}

Let $y_n \in \BR^d$ be the closest point on $\pa D_n$ to $x$: $\norm{x - y_n} = \dist(x, \pa D_n)$. 

\begin{lemma} The sequence $(y_n)_{n \ge 1}$ is bounded.
\label{lemma:y-n-bounded}
\end{lemma}

\begin{proof} Consider three cases:

\medskip

{\it Case 1:} $x \in D_0$. Then $x \in D_0 \subseteq \varliminf\limits_{n \to \infty}D_n$. 
Therefore, $x \in D_n$ for $n \ge n_x$. Now, take any $y \in \ol{D}_0^c$; then 
$y \in \ol{D}_0^c \subseteq \varliminf\limits_{n \to \infty}\ol{D}_n^c$. Therefore, $y \in D_n$ for $n \ge n_y$. Take $n \ge n_x\vee n_y$; then 
$x \in D_n$ and $y \in \ol{D}_n^c$. Therefore, $[x, y]\cap \pa D_n \ne \varnothing$. Take some $u_n \in [x, y]\cap \pa D_n$; then $\norm{x - y_n} \le \dist(x, \pa D_n) \le \norm{x - u_n} \le \norm{x - y}$. Thus, $\norm{y_n} \le \norm{x} + \norm{x - y}$. 

\medskip

{\it Case 2:} $x \in \ol{D}_0^c$. This is similar to Case 1. 

\medskip

{\it Case 3:} $x \in \pa D_0$. Use Lemma~\ref{lemma:tending} below and find a sequence $x_n \in \pa D_n$ such that $x_n \to x$. Then $\norm{x - y_n} = \dist(x, \pa D_n) \le \norm{x - x_n} \to 0$. Therefore, $y_n \to x$, and $(y_n)_{n \ge 1}$ is bounded. 
\end{proof}

Let us show that 
\begin{equation}
\label{eq:sup-limit}
\dist(x, \pa D_0)  \le \varliminf\limits_{n \to \infty}\dist(x, \pa D_n).
\end{equation}
Take a subsequence $(n_k)_{k \ge 1}$. It suffices to show that there exists a subsequence $(n'_k)_{k \ge 1} \subseteq (n_k)_{k \ge 1}$ such that 
$$
\dist(x, \pa D_0) \le \lim\limits_{k \to \infty}\dist\bigl(x, \pa D_{n'_k}\bigr).
$$
The sequence $(y_{n_k})_{k \ge 1}$ is bounded by Lemma~\ref{lemma:y-n-bounded}. Therefore, there exists a subsequence $(n'_k)_{k \ge 1} \subseteq (n_k)_{k \ge 1}$ such that $y_{n'_k} \to \ol{y}$. By assumption (vii), $\ol{y} \in \pa D_0$. Therefore, 
$$
\dist(x, \pa D_0) \le \norm{x - \ol{y}} = \lim\limits_{k \to \infty}\norm{x - y_{n'_k}} = 
\lim\limits_{k \to \infty}\dist\bigl(x, \pa D_{n'_k}\bigr).
$$
This proves~\eqref{eq:sup-limit}. Now, let us show that
\begin{equation}
\label{eq:inf-limit}
\dist(x, \pa D_0) \ge \varliminf\limits_{n \to \infty}\dist(x, \pa D_n).
\end{equation}
By Lemma~\ref{lemma:tending}, there exists a sequence $\ol{y}_n \in \pa D_n$ such that $\ol{y}_n \to y_0$. Therefore, $\dist(x, \pa D_0) = \norm{x - y_0} = \lim\limits_{n \to \infty}\norm{x - \ol{y}_n}$. But $\norm{x - \ol{y}_n} \le \dist(x, \pa D_n)$. This proves~\eqref{eq:inf-limit}. 

\subsection{Proof of Lemma~\ref{lemma:aux}}  We need only to prove continuity of $\zeta$, the rest is done in \cite[Lemma 3.2]{MyOwn8}. Let $x_n \to x_0$ in $\CK_0$, and take $y_0$, a limit point of $\zeta(x_n)$. Without loss of generality assume 
$y_0 = \lim\limits_{n \to \infty}\zeta(x_n)$. Then 
$\dist(x_n, \pa D_0) = \norm{x_n - \zeta(x_n)} \to \norm{x_0 - y_0}$. 
But the distance function is continuous. So $\norm{x_0 - y_0} = \dist(x_0, \pa D_0)$. Since the closest point on $\pa D_0$ to $x_0$ is unique, we have: $y_0 = \zeta(x_0)$. The proof is complete.

\subsection{Proof of Corollary~\ref{cor:simple}} 

(i) The proof is trivial. 

(ii) Let us prove the first statement, when $\CK \subseteq D_0$; the second one is similar. From Lemma~\ref{lemma:equiv} (iii) we have: 
$\phi_n(x) \to \phi_0(x) > 0$ uniformly on $\CK$, and $\phi_0$ is continuous on $\CK$. Therefore, there exists $\eps > 0$ such that $\phi_0(x) \ge \eps$ for $x \in \CK$. By the uniform convergence, there exists $n_0$ such that for $n > n_0$ we have: $\phi_n(x) \ge \eps/2 > 0$ for $x \in \CK$. This completes the proof.

\subsection{Proof of Lemma~\ref{lemma:monotone}} Let us show the first case, when $D_n \uparrow D_0$; the second case is similar. 

{\bf Case 1. $x \in D$.} Then $\phi_0(x) =: r > 0$. There exists $n_x$ such that for $n \ge n_x$ we have: $x \in D_n$. Therefore, $\phi_n(x) > 0$ for $n \ge n_x$, and $\phi_n(x) = \dist(x, \pa D_n) = \dist(x, D_n^c)$. We have: $(\phi_n(x))_{n \ge n_x}$ is a nondecreasing sequence, and $\phi_n(x) \le \phi_0(x)$ for each $n \ge n_x$. 

Now, fix $\eps > 0$ and consider the closed ball $B(x, r - \eps) \subseteq D$. We have: $B(x, r - \eps) \subseteq \cup D_n$. But this ball is compact, so there exists a finite subcover $D_{n_1}, \ldots, D_{n_m}$. Take $k_x := \max(n_1, \ldots, n_m, n_x)$. Then $B(x, r - \eps) \subseteq D_{k_x}$. Therefore, $\phi_{k_x}(x) \ge r - \eps$. By monotonicity of $(\phi_n(x))_{n \ge n_x}$, we have: 
$\phi_n(x) \ge r - \eps = \phi_0(x) - \eps$ for $n \ge k_x$. Since $\eps > 0$ is arbitrary, this proves that $\phi_n(x) \to \phi_0(x)$ as $n \to \infty$. 

\medskip

{\bf Case 2. $x \notin D$.} Then $\phi_0(x) \le 0$. Therefore, $\dist(x, \pa D_0) = |\phi_0(x)|$. Take $y \in \pa D_0$ such that $\norm{y - x} = |\phi_0(x)|$. Fix $\eps > 0$; then there exists $z \in D$ such that $\norm{z - y} \le \eps$.  Therefore, $\norm{z - x} \le \norm{z - y} + \norm{y - x} \le |\phi_0(x)| + \eps$. Because $D = \cap D_n$, there exists $n_0$ such that $z \in D_n$ for $n \ge n_0$. Therefore, $\dist(x, D_n) \le |\phi_0(x)| + \eps$. But $x \notin D_n$ for all $n \ge 1$; therefore, $-\dist(x, D_n) = \phi_n(x)$. But $\dist(x, D_n) \le \norm{x-z} \le |\phi_0(x)| + \eps$. Therefore, 
\begin{equation}
\label{eq:conv-of-phis}
\phi_n(x) \ge -|\phi_0(x)| - \eps = \phi_0(x) - \eps\ \ \mbox{for}\ \ n \ge n_0.
\end{equation}
But $D_n \uparrow D_0$, and $x \notin D_0$. Therefore, $(\dist(x, D_n))_{n \ge 1}$ is nonincreasing, and so $(\phi_n(x) = -\dist(x, D_n))_{n \ge 1}$ is nondecreasing. Also, $\dist(x, D_n) \ge \dist(x, D_0)$, and so 
$$
\phi_n(x) = -\dist(x, D_n) \le -\dist(x, D_0) = \phi_0(x).
$$
Therefore, $(\phi_n(x))_{n \ge 1}$ is a nonincreasing sequence, bounded below by $\phi_0(x)$. Together with~\eqref{eq:conv-of-phis}, this gives $\phi_n(x) \to \phi_0(x)$ as $n \to \infty$. 

\subsection{Proof of Lemma~\ref{lemma:comparison-conv}} (i) The fact that $D_n \Ra D_0$ implies Wijsman convergence follows from Lemma~\ref{lemma:equiv} (ii). Now, let us give a counterexample which shows that weak convergence does not coincide with Wijsman convergence. Take the following sequence of domains in $\BR^2$:
$$
D_n :=\Int\left[\left(\BR\times\BR_+\right)\setminus\left([2^{-n-1}, 2^{-n}]\times[0, 1]\right)\right],\ \ n = 1, 2, \ldots,
$$
and the limiting domain $D_0 = \BR\times(0, \infty)$. Then $D_n \to D_0$ in Wijsman topology, but not in the weak sense. Indeed, for $x_0 = (0, 1)'$ we have: $\phi_n(x_0) \le 2^{-n-1}$, because the distance from $x_0$ to the boundary $\pa D_n$ is less than or equal to the distance to the point $(2^{-n-1}, 1)'$ on the boundary. But $\phi_0(x_0) = 1$, because the distance from $x_0$ to the boundary $\pa D_0$ (which is the $x_1$-axis) is equal to $1$. This contradicts that $\phi_n(x_0) \to \phi_0(x_0)$. 

(ii) Now, Hausdorff convergence implies weak convergence: if $D_n \to D_0$ in Hausdorff sense, then $D_n \to D_0$ in Wijsman sense, but also $D_n^c \to D_0^c$ in Hausdorff sense, so $D_n^c \to D_0^c$ in Wijsman sense; use Lemma~\ref{lemma:equiv}(ii) and complete the proof. 

But weak convergence does not imply Hausdorff convergence. Indeed, let $d = 2$ and 
$D_0 := \BR^+_2$, and $D_n$ be the result of rotation of $D_0$ counterclockwise by angle $\al_n$ around the origin, where $\al_n \to 0$. Then $D_n \Ra D_0$ (this is a particular case of Theorem~\ref{thm:SRBM} below), but not $D_n \to D_0$ in Hausdorff sense. 

(iii) If $\phi_n(x) \to \phi_0(x)$ uniformly on $\BR^d$, then $(\phi_n(x))_- = \dist(x, D_n) \to \dist(x, D_0) = (\phi_0(x))_-$ uniformly on $\BR^d$. Therefore, $D_n \to D_0$ in Hausdorff topology. Conversely, if $D_n \to D_0$ in Hausdorff topology, then $D_n^c \to D_0^c$ in Hausdorff topology, and 
$\dist(x, D_n) \equiv (\phi_n(x))_- \to \dist(x, D_0) \equiv (\phi_0(x))_-$, $\dist(x, D_n^c) \equiv (\phi_n(x))_+ \to \dist(x, D_0^c) \equiv (\phi_0(x))_+$, 
uniformly on $\BR^d$. Adding these convergence relations and noting that $a \equiv a_+ + a_-$ for $a \in \BR^d$, we complete the proof. 

\medskip

{\it Proof of Lemma~\ref{lemma:loc-unif-conv}.} {\bf (i) $\Ra$ (ii).} Without loss of generality, assume $n_k = k$. Take a sequence $(x_n)_{n \ge 0}$ of functions, as described in Lemma~\ref{lemma:loc-unif-conv}. Assume that $f_n(x_n)$ does not converge to $f_0(x_0)$ uniformly. Then there exists $\eps > 0$, a subsequence $(m_k)_{k \ge 1}$, and a sequence $(t_{m_k})_{k \ge 1}$ in $[0, T]$ such that 
\begin{equation}
\label{eq:777}
\bigl|f_{m_k}\bigl(x_{m_k}\bigl(t_{m_k}\bigr)\bigr) - f_0\bigl(x_0\bigl(t_{m_k}\bigr)\bigr)\bigr| \ge \eps.
\end{equation}
We can extract a convergent subsequence $t_{m'_k} \to t_0 \in [0, T]$. Then 
$$
x_{m'_k}\left(t_{m'_k}\right) \to x_0(t_0),\ \ \mbox{and}\ \ x_0\bigl(t_{m'_k}\bigr) \to x_0(t_0).
$$
Therefore, since $f_n \to f_0$ locally uniformly and $f_0$ is continuous on $E_0$, 
$$
f_{m'_k}\bigl(x_{m_k}\bigl(t_{m_k}\bigr)\bigr) \to f_0(x_0(t_0)),\ \ \mbox{and}\ \ 
f_0\bigl(x_0\bigl(t_{m_k}\bigr)\bigr) \to f_0(x_0(t_0)).
$$
This contradicts~\eqref{eq:777}. 

\medskip

{\bf (ii) $\Ra$ (i).} Take $x_{n_k}(t) \equiv z_{n_k}$ and $x_0(t) \equiv z_0$. 

\subsection{Proof of Lemma~\ref{lemma:conv-of-non-smooth}} Assume the condition (b) holds. Take a sequence $x_{n_k} \in \CV_{n_k}$ such that $x_{n_k} \to x_0$. Since the set $\CK := \ol{\{x_{n_k}\mid k = 1, 2,\ldots\}}$ is compact, from condition (b) we have: $\dist\bigl(x_{n_k}, \CV_0\bigr) \to 0$. 
The function $\dist(\cdot, \CV_0)$ is continuous. Therefore, $\dist(x_0, \CV_0) = 0$, which means $x_0 \in \CV_0$ (because the set $\CV_0$ is closed). Conversely, assume that for every sequence $(x_{n_k})_{k \ge 1}$ such that $x_{n_k} \in \CV_{n_k}$ and $x_{n_k} \to x_0$ we have: $x_0 \in \CV_0$. Take a compact set 
$\CK \subseteq \BR^d$. Let us show~\eqref{eq:condition-b}. Assume the converse: there exists $\eps > 0$ and a subsequence $(n_k)_{k \ge 1}$ such that 
$$
\max\limits_{x \in \CV_{n_k}\cap\CK}\dist(x, \CV_0) > \eps.
$$
Then there exists $x_{n_k} \in \CV_{n_k}\cap\CK$ such that $\dist(x_{n_k}, \CV_0) > \eps$. Now, the sequence $(x_{n_k})_{k \ge 1}$ is bounded, so there exists a limit point $\ol{x} := \lim x_{n'_k}$. Therefore, $\ol{x} \in \CV_0$ by our assumption. And
$$
\dist(\ol{x}, \CV_0) = \lim\limits_{k \to \infty}\dist(x_{n'_k}, \CV_0) \ge \eps.
$$
This contradiction completes the proof. 

\subsection{Proof of Lemma~\ref{lemma:conv-to-Rd}} Let us show the ``only if'' part. 
Take a compact set $\CK \subseteq \BR^d$ and assume that there exists a subsequence $(n_k)_{k \ge 1}$ such that for some $x_{n_k} \in \CK$, we have: $x_{n_k} \notin D_{n_k}$. Extract a convergent subsequence: $x_{n'_k} \to y \in \CK$. 
We claim that 
\begin{equation}
\label{eq:varlimsuple0}
\varlimsup\limits_{k \to \infty}\phi_{n'_k}(y) \le 0.
\end{equation}
Indeed, if $y \notin D_{n'_k}$, then $\phi_{n'_k}(y) \le 0$. If $y \in D_{n'_k}$, then 
$\phi_{n'_k}(y) = \dist\bigl(y, \pa D_{n'_k}\bigr) = \dist\bigl(y, D^c_{n'_k}\bigr) \le \norm{y - x_{n'_k}} \to 0$. This proves the claim~\eqref{eq:varlimsuple0}. But this contradicts the assumption that $\phi_n(y) \to \infty$. 

Now, let us show the ``if'' part. Take $x \in \BR^d$ and let $\CK := \ol{U(x, N)}$ for large $N$. By assumption, there exists $n_0$ such that for $n > n_0$ we have: $\CK \subseteq D_n$. So $x \in D_n$, and $\phi_n(x) = \dist(x, \pa D_n) \ge N$. Since $N$ is arbitrarily large, this completes the proof.

\section*{Acknowledgements}

The author would like to thank \textsc{Ioannis Karatzas} and \textsc{Ruth Williams} for help and useful discussion. The author would also like to express his gratitude to two anonymous referees, who pointed out many mistakes and misprints. This research was partially supported by  NSF grants DMS 1007563, DMS 1308340, DMS 1405210, DMS 1409434.


\medskip\noindent

\begin{thebibliography}{12}


\bibitem{Wijsman1994} \textsc{Gerald Beer} (1994). Wijsman Convergence: a Survey. \textit{Set-Valued Anal.} \textbf{2} (1-2), 77-94.

\bibitem{BurdzyChen1998} \textsc{Krzysztof Burdzy, Zhen-Qing Chen} (1998). Weak Convergence of Reflected Brownian Motions. \textit{Electr. J. Probab.} \textbf{3} (4), 29-33. 

\bibitem{Marshall} \textsc{Krzysztof Burdzy, Donald Marshall} (1992). Hitting a Boundary Point with Reflected Brownian Motion. Springer, \textit{Lecture Notes in Math.} \textbf{1526}, 81-94.  

\bibitem{DW1995} \textsc{Jim G. Dai, Ruth J. Williams} (1995). Existence and Uniqueness of Semimartingale Reflecting Brownian Motions in Convex Polyhedra. \textit{Th. Probab. Appl.} \textbf{40} (1), 3-53. 


\bibitem{Delfour1994} \textsc{Michel C. Delfour, Jean-Paul Zolesio} (1994). Shapre Analysis via Oriented Distance Functions. \textit{J. Funct. Anal.} \textbf{123} (1), 129-201. 

\bibitem{HR1981a} \textsc{J. Michael Harrison, I. Martin Reiman} (1981). Reflected Brownian Motion on an Orthant. \textit{Ann. Probab.} \textbf{9} (2), 302-308. 

\bibitem{HornBook} \textsc{Roger A. Horn, Charles R. Johnson} (1990). \textit{Matrix Analysis.} Cambridge University Press. 

\bibitem{Thesis} \textsc{Scott Hottovy} (2013). The Smoluchowski-Kramers Approximation for Stochastic Differential Equations with Arbitrary State-Dependent Friction. \textit{Ph.D. Thesis}. 

\bibitem{IWBook} \textsc{Nobuyuki Ikeda, Shinzo Watanabe} (1989). \textit{Stochastic Differential Equations and Diffusion Processes.} North-Holland. 

\bibitem{KR2014} \textsc{Weining Kang, Kavita Ramanan} (2014). On The Submartingale Problem for Reflected Diffusions in Domains with Piecewise Smooth Boundaries. Available at arXiv:1412:0729. 

\bibitem{KangWilliams2007} \textsc{Weining Kang, Ruth J. Williams} (2007). An Invariance Principle for Semimartingale Reflecting Brownian Motions in Domains with Piecewise Smooth Boundaries. \textit{Ann. Appl. Probab.} \textbf{17} (2), 741-779. 

\bibitem{KurtzProtter1991} \textsc{Thomas G. Kurtz, Philip Protter} (1991). Weak Limit Theorems for Stochastic Integrals and Stochastic Differential Equations. \textit{Ann. Probab.} \textbf{19} (3), 1035-1070. 

\bibitem{MunkresBook} \textsc{James R. Munkres} (1975). \textit{Topology: a First Course.} Prentice-Hall. 

\bibitem{Ramasubramanian1983} \textsc{S. Ramasubramanian} (1983). Recurrence of Projections of Diffusions. \textit{Sankhy A} \textbf{45} (1), 20-31. 

\bibitem{Ramasubramanian1988} \textsc{S. Ramasubramanian} (1988). Hitting of Submanifolds by Diffusions. \textit{Probab. Th. Rel. Fields} \textbf{78} (1), 149-163. 

\bibitem{RW1988} \textsc{I. Martin Reiman, Ruth J. Williams} (1988). A Boundary Property of Semimartingale Reflecting Brownian Motions. \textit{Probab. Th. Rel. Fields} \textbf{77} (1), 87-97. 

\bibitem{MyOwn6} \textsc{Andrey Sarantsev} (2016). Infinite Systems of Competing Brownian Particles. Available at arXiv:1403.4229. 

\bibitem{MyOwn8} \textsc{Andrey Sarantsev} (2016). Penalty Method for Obliquely Reflected Diffusions. Available at arXiv:1509.01777.

\bibitem{MyOwn3} \textsc{Andrey Sarantsev} (2015). Triple and Simultaneous Collisions of Competing Brownian Particles. \textit{Electr. J. Probab.} \textbf{20} (1), 1-29. 

\bibitem{TW1993} \textsc{Lisa M. Taylor, Ruth J. Williams} (1993). Existence and Uniqueness of Semimartingale Reflecting Brownian Motions in an Orthant. \textit{Probab. Th. Rel. Fields} \textbf{75} (4), 459-485. 

\bibitem{Wijsman1966} \textsc{Robert A. Wijsman} (1966). Convergence of Sequences of Convex Sets, Cones and Functions II. \textit{Trans. Amer. Math. Soc.} \textbf{123},  32-45. 

\bibitem{Wil1987} \textsc{Ruth J. Williams} (1987). Reflected Brownian Motion with Skew-Symmetric Data in a Polyhedral Domain. \textit{Probab. Th. Rel. Fields} \textbf{75} (4), 459-485. 

\bibitem{Wil1995} \textsc{Ruth J. Williams} (1995). Semimartingale Reflecting Brownian Motions in the Orthant. Springer, \textit{IMA Vol. Math. Appl.} \textbf{71}, 125-137.

\bibitem{Wil1998b} \textsc{Ruth J. Williams} (1998). An Invariance Principle for Semimartingale Reflecting Brownian Motions in an Orthant. \textit{Queueing Syst. Th. Appl.} \textbf{30} (1-2), 5-25.


\end{thebibliography}
\end{document}